\documentclass[12pt]{amsart}
\usepackage{txfonts}
\usepackage{mathrsfs}
\usepackage{graphicx}

\usepackage{amsfonts}
\usepackage{amsthm}
\usepackage{graphics}
\usepackage{indentfirst}
\usepackage{cite}
\usepackage{latexsym}
\usepackage{amsmath}
\usepackage{amssymb}
\usepackage[dvips]{epsfig}
\usepackage{amscd}

\newtheorem{theorem}{Theorem}[section]
\newtheorem{remark}{Remark}[section]

\newtheorem{lemma}[theorem]{Lemma}

\newcommand{\bt}{\begin{theorem}}
\newcommand{\bl}{\begin{lemma}}
\newcommand{\el}{\end{lemma}}
\newcommand{\et}{\end{theorem}}

\newcommand{\bn}{\begin{eqnarray}}
\newcommand{\en}{\end{eqnarray}}
\newcommand{\bnn}{\begin{eqnarray*}}
\newcommand{\enn}{\end{eqnarray*}}

\newcommand{\ba}{\begin{aligned}}
\newcommand{\ea}{\end{aligned}}
\newcommand{\be}{\begin{equation}}
\newcommand{\ee}{\end{equation}}

\def\norm[#1]#2{\|#2\|_{#1}}

\newcommand{\Bn}{{\boldsymbol{n}}}
\newcommand{\Bu}{{\boldsymbol{u}}}
\newcommand{\BJ}{{\boldsymbol{J}}}
\newcommand{\BLambda}{{\boldsymbol{\Lambda}}}
\newcommand{\Bsigma}{{\boldsymbol{\sigma}}}
\newcommand{\Bepsilon}{{\boldsymbol{\epsilon}}}
\newcommand{\Bv}{{\boldsymbol{v}}}
\newcommand{\Bh}{{\boldsymbol{h}}}
\newcommand{\Bw}{{\boldsymbol{w}}}
\newcommand{\Bz}{{\boldsymbol{z}}}
\newcommand{\Bpsi}{{\boldsymbol{\psi}}}
\newcommand{\Bxi}{{\boldsymbol{\xi}}}
\newcommand{\Bchi}{{\boldsymbol{\chi}}}
\newcommand{\Bvarphi}{{\boldsymbol{\varphi}}}
\newcommand{\Bphi}{{\boldsymbol{\phi}}}
\newcommand{\Btau}{{\boldsymbol{\tau}}}

\makeatletter      % '@' is now a normal "letter" for TeX
\@addtoreset{equation}{section}
\makeatother       % '@' is restored as a "non-letter" character for TeX

\begin{document}

\title
{Solvability of generalized div-curl system and Friedrichs inequalities}

\author{Yun Wang}
\address{Department of Mathematics, Soochow University, SuZhou\\
215006, People's Republic of China}
\email{ywang3@suda.edu.cn}
\date{}

\maketitle

\begin{abstract}
In this article, we consider the solvability of two generalized {\bf div-curl} systems. They are referred to as the equations of magnetostatics and electro- , resp.. Necessary compatibility conditions on the data for the existence of solutions are fully described, and existence of $W^{m, p}$-solutions is proved. Moreover, we give some description for the null spaces. As a corollary, we give the estimates of gradient  via generalized ${\rm div}$ and ${\rm curl}$ in $L^p$-framework, which can be considered as Friedrichs inequalities. Furthermore, two decompositions of Sobolev spaces are given.


AMS: 35F15, 78A30, 35D35
\end{abstract}

%%%%%%%%%%%%%%%%%%%%%%%%%%%%%%%%%%%%%%%%%%%%%%%%%%%%%%%%%%%%%%%%Introduction%%%%%%%%%%%%%%%%%%%%%%%%%%%%%%%%%%%%%%%%%%%%

\section{Introduction}

In this paper, we consider the solvability of the following systems of first order differential equations,
\begin{equation}\label{1-1}\left\{
\begin{array}{l}
{\rm curl}~(\Bsigma \Bu ) = \BJ, \ \ \ \mbox{in}\ \Omega,\\[2mm]
{\rm div}~\Bu = \rho, \ \ \ \ \ \mbox{in}\ \Omega, \\[2mm]
\Bu \cdot \Bn = \lambda ,\ \ \ \ \mbox{on}\ \partial \Omega,
\end{array}
\right.
\end{equation}
and
\begin{equation}\label{1-2} \left\{
\begin{array}{l}
{\rm curl}~\Bu = \BJ , \ \ \ \mbox{in} \ \Omega, \\[2mm]
{\rm div}~(\Bepsilon \Bu) = \rho, \ \ \ \mbox{in}\ \Omega, \\[2mm]
\Bu\times \Bn = \BLambda,\ \ \ \ \mbox{on}\ \partial \Omega.
\end{array}
\right.
\end{equation}
Here $\Omega \subset \mathbb{R}^3$ is a bounded domain with smooth boundary $\partial \Omega$. $\Bn$ is the outer unit normal on $\partial \Omega$. The coefficients  $\Bsigma, \Bepsilon$ are $3\times 3$ symmetric matrices, with smooth elements, satisfying the uniform ellipticity condition, i.e.,
there exist positive constants $m, M$ such that
\begin{equation}\nonumber
m|\xi|^2 \leq \Bsigma(x) \xi \cdot \bar{\xi} , \ \ \Bepsilon (x) \xi \cdot \bar{\xi} \leq M |\xi|^2, \ \ \ \mbox{for all} \ \ \xi \in \mathbb{C}^3, \ \ x\in \Omega.
\end{equation}
 ${\rm curl}$, ${\rm div}$ are the vorticity operator and divergence operator respectively.

The above problem is one of the most fundamental linear systems in linear physics. It can be found in eletromagnetism. As noted in \cite{Picard, Saranen82}, the two systems are closely  related to the time-harmonic Maxwell equations for inhomogeneous anisotropic medium. System \eqref{1-1} with $\rho =0$  describes a magnetic field, while System \eqref{1-2} with $\BJ =0$ describes a static electric field. Hence the two systems are referred to as the equations of magnetostatics and electro-  repectively.  When $\Bsigma, \Bepsilon$ are identity matrices, the system can also be found in fluid mechanics. For example, the solvability of \eqref{1-1} can help us recover the velocity from the knowledge of vorticity.

Let us give a brief review of prior results. 
The classical results date back to \cite{ Duff, DS, Friedrichs, Gaffney}. The problem was studied in generalized setting of 
alternating differential forms on Riemannian manifolds. Picard\cite{Picard} and Saranen\cite{Saranen82, Saranen83} initiated the  research in Hilbert framework. They both proved the existence of solutions when $(\BJ, \rho, \lambda) \in L^2(\Omega) \times L^2(\Omega) \times H^{-\frac12}(\partial \Omega)$ for \eqref{1-1} and $(\BJ, \rho, \BLambda) \in L^2(\Omega) \times L^2(\Omega) \times H^{-\frac12}(\partial \Omega)$ for \eqref{1-2}. Inspired by the  related results in the  setting of differential forms\cite{Duff, DS, Friedrichs, Gaffney}, they analyzed the null spaces, which may be not zero, due to the topological structure of $\Omega$. They  gave the dimension of null spaces,  which are proved to be related with Betti numbers and independent of $\Bsigma$ and $\Bepsilon$. 

When $\Bsigma$ and $\Bepsilon$ are identity matrices, the solvability of \eqref{1-1} and \eqref{1-2}($\rho = 0$, $\BLambda = 0$) in $W^{m, p}$-spaces has been well studied, see \cite{AV, AS, KY}. For some other solvability results on Sobolev(Besov)-type domains, please refer to \cite{CS, MP}. However, the solvability in $W^{m, p}$-spaces for the general case, i. e., $\Bsigma$ and 
$\Bepsilon$ are symmetric positive definite matrices, is not clear. That is one goal of our paper. 

In addition to the solvability results, some estimates of solutions were derived. An inequality of Friedrichs\cite{Friedrichs} states that 
\begin{equation}\label{Friedrichs}
\| \Bu\|_{H^1(\Omega)} \leq C \left(   \|{\rm curl}~\Bu\|_{L^2(\Omega)} + \|{\rm div}~(\Bepsilon \Bu)\|_{L^2(\Omega)} + \| \Bu \|_{L^2(\Omega)}      \right),
\end{equation}
for all vector fields $\Bu \in L^2(\Omega)$ satisfying $\Bu \times \Bn =0$ on $\partial \Omega$. For the case that $\Bsigma$ and 
$\Bepsilon$ are identity matrices, similar inequalities were then derived in  $L^p$-spaces\cite{Wahl} and $C^\alpha$-spaces\cite{BW}. Recently, \cite{AS, KY} gave a unified version, 
\begin{equation}\label{Sobolev1}
\|\Bu\|_{W^{m, p}(\Omega)} \leq C \left(   \| \Bu\|_{L^p(\Omega)} + \|{\rm curl }~ \Bu \|_{W^{m-1, p}(\Omega)} + \|{\rm div}~ \Bu\|_{W^{m-1, p}(\Omega)} +
\| \Bu \cdot \Bn\|_{W^{m-\frac1p, p}(\partial \Omega)}               \right).
\end{equation}
\begin{equation}\label{Sobolev2}
\| \Bu \|_{W^{m, p}(\Omega)} \leq C \left(   \| \Bu\|_{L^p(\Omega)} + \|{\rm curl }~\Bu \|_{W^{m-1, p}(\Omega)} + \|{\rm div}~\Bu)\|_{W^{m-1, p}(\Omega)} +
\| \Bu \times \Bn\|_{W^{m-\frac1p, p}(\partial \Omega)}    \right).
\end{equation}
These inequalities give the estimates of $\nabla \Bu$ via ${\rm div}~\Bu$, ${\rm curl}~\Bu$ and boundary values. And they played an important role in the studies of  Navier-Stokes equations, Euler equations. 

When we study the regularity of solutions to Maxwell equations, we find that Friedrichs inequalities as \eqref{Sobolev1}-\eqref{Sobolev2} for the general case(i. e., $\Bsigma$ and $\Bepsilon$ are symmetric positive definite matrices) are required. Shen-Song\cite{SS} derived the following inequality on simply connected domain $\Omega$,
\begin{equation}\label{1-19}
\|\nabla \Bu\|_{L^p(\Omega)} \leq C \left( \|{\rm curl }~\Bu\|_{L^p(\Omega)} + \|{\rm div}~(\Bepsilon \Bu)\|_{L^p(\Omega)} +
\| \Bu \times \Bn\|_{W^{1-\frac1p, p}(\partial \Omega)}   \right).
\end{equation}
We try to get a more general version. That is another goal of our paper. 

There are also some generalizations to the mixed boundary value problem, see \cite{AA1, AA2}.

In this paper, we will show the solvability of \eqref{1-1} and \eqref{1-2} in $W^{m, p}$-spaces . Necessary and suffcient compatibility conditions  will be fully described.  As remarked above, the null spaces may be not zero. The fact will bring some difficulties. So we will give some further discussion for the null spaces. In particular, explicit bases of the two null spaces in $L^p$-framework are constructed.   

Some regularity estimates of solutions will be given in this paper. As a corollary, some  inequalities of Friedrichs type are derived:
\begin{equation}\label{1-20}
\|\Bu\|_{W^{m, p}(\Omega)} \leq C \left( \| \Bu\|_{L^p(\Omega)} + \|{\rm curl }~(\Bsigma \Bu )\|_{W^{m-1, p}(\Omega)} + \|{\rm div}~\Bu\|_{W^{m-1, p}(\Omega)} +
\| \Bu \cdot \Bn\|_{W^{m-\frac1p, p}(\partial \Omega)}   \right),
\end{equation}
and
\begin{equation}\label{1-21}
\| \Bu\|_{W^{m, p}(\Omega)} \leq C \left( \|\Bu\|_{L^p(\Omega)} + \|{\rm curl }~\Bu \|_{W^{m-1, p}(\Omega)} + \|{\rm div}~(\Bepsilon \Bu)\|_{W^{m-1, p}(\Omega)} +
\| \Bu \times \Bn\|_{W^{m-\frac1p, p}(\partial \Omega)}   \right).
\end{equation}
The above inequalities give the estimates of gradient via generalized ${\rm div}$, ${\rm curl}$ and boundary values. They generalize the inequalities \eqref{Sobolev1}-\eqref{Sobolev2}. 
As showed in \cite{SS}, the inequalities \eqref{1-20}- \eqref{1-21} may help to get the regularity estimates for solutions to Maxwell equations.

%In addition, we will apply the inequalities to derive some new regularity
%result for time-harmonic Maxwell equations with impedence boundary condition.

This paper is organized as follows: In Section 2, we will  state our main theorems and give some notations. In Section 3 and Section 4, solvability of problems \eqref{1-1} and \eqref{1-2} are proved respectively. Some related inequalities of Friedrichs type are derived. In Section 5, we give two decompositions of Sobolev spaces, which are designed for solvability of Maxwell equations.

%%%%%%%%%%%%%%%%%%%%%%%%%%%%%%%%%Preliminaries%%%%%%%%%%%%%%%%%%%%%%%%%%%%%%%%%%%%%%%%%%%%%%

\section{Preliminaries and Main Reuslts}
For simplicity of writing, we assume that $\Omega$ is a bounded domain in $\mathbb{R}^3$ with  $C^\infty$-boundary $\partial \Omega$, and

(1)\ \ $\partial \Omega$ consists of $(N_1+1)$ components $\Gamma_i$, $0\leq i \leq N_1$. $\Gamma_0$ is the boundary of the only unbounded connected component of $\mathbb{R}^3\setminus \overline{\Omega}$. 

(2)\ \ We do not assume that $\Omega$ is simply connected. There are $N_2$ $C^\infty$-surfaces $\Sigma_1, \cdots, \Sigma_{N_2}$, transversal to $\partial \Omega$ such that
\begin{equation}\nonumber
\Sigma_i \cap \Sigma_j = \varnothing,\ \ \ i \neq j,
\end{equation} and 
\begin{equation}\nonumber \displaystyle
\Omega^0= \Omega \setminus  \cup_{j =1}^{N_2} \Sigma_j \ \mbox{is a simply connected domain.}
\end{equation}

Before stating our main results,  let us introduce some notations. Let $L^p(\Omega)$ denote the usual scalar-valued and vector-valued $L^p$-space over $\Omega$, $1\leq p \leq \infty.$ Let
$$W^{m, p}(\Omega)= \{u \in L^{p}(\Omega):\ D^\alpha u \in L^p(\Omega),\ |\alpha| \leq m\},\ \ \ m \in \mathbb{N}.$$

Define the spaces:
$$L^p({\rm div}; \Omega) = \{ \Bu \in L^p(\Omega):\ {\rm div}~\Bu \in L^p(\Omega)   \},$$
$$L^p({\rm curl}; \Omega) = \{ \Bu \in L^p(\Omega):\ {\rm curl}~\Bu \in L^p(\Omega)  \}.$$

For every function $\Bv\in L^p( {\rm div}; \Omega)$, we denote $\Bv\cdot \Bn$ the normal boundary value of $\Bv$ defined in $W^{-\frac1p, p}(\partial \Omega)$,
\begin{equation}\nonumber
\forall \ \varphi \in W^{1, p^{\prime} }(\Omega), \ \ \ <\Bv\cdot \Bn,\ \varphi >_{\partial \Omega}
= \int_{\Omega}  \Bv \cdot \nabla \varphi \, dx + \int_{\Omega} {\rm div }~\Bv \cdot \varphi \, dx.
\end{equation}
And it holds that
\begin{equation}\label{normal}
\| \Bv \cdot \Bn \|_{W^{-\frac1p, p}(\partial \Omega)} \leq C \left(  \| \Bv\|_{L^p(\Omega)} + \| {\rm div}~\Bv\|_{L^p(\Omega)}           \right).
\end{equation}

For every function $\Bv\in L^p( {\rm curl}; \Omega)$, we denote $\Bv\times \Bn$ the tangential boundary value of $\Bv$ defined in $W^{-\frac1p, p}(\partial \Omega)$,
\begin{equation}\nonumber
\forall \ \Bvarphi \in W^{1, p^{\prime}}(\Omega), \ \ \ <\Bv\times \Bn,\ \Bvarphi >_{\partial \Omega}
= \int_{\Omega}  \Bv \cdot {\rm curl }~\Bvarphi \, dx - \int_{\Omega} {\rm curl }~\Bv \cdot \Bvarphi \, dx.
\end{equation}
And it holds that
\begin{equation}\label{tangential}
\| \Bv \times \Bn \|_{W^{-\frac1p, p}(\partial \Omega)} \leq C \left(  \| \Bv\|_{L^p(\Omega)} + \| {\rm curl}~\Bv\|_{L^p(\Omega)}           \right).
\end{equation}

Denote the tangential part of $\Bv$ on the boundary,
\begin{equation}\nonumber \Bv_T = \Bn \times (\Bv \times \Bn) = \Bv - (\Bv\cdot \Bn ) \Bn.
\end{equation}
And denote the tangential gradient of $\psi$ on the boundary,
\begin{equation}\nonumber
\nabla_T \psi = \nabla \psi- (\nabla \psi \cdot \Bn) \Bn .
\end{equation}
If $\boldsymbol{f}\in L^p(\partial \Omega)$ satisfies $\boldsymbol{f} \cdot \Bn =0$ on $\partial \Omega$,
we will use ${\rm div}_T\ ( \boldsymbol{f} )$ to denote the surface divergence of $ \boldsymbol{f}$ on $\partial \Omega$, defined by
\begin{equation}\nonumber  \forall \ \psi \in W^{1, p^{\prime}}(\partial \Omega) , \ \ \ 
<{\rm div}_T\ \boldsymbol{f},\ \psi>_{W^{-1, p}(\partial \Omega) \times W^{1, p^{\prime}}(\partial \Omega)}
= - \int_{\partial \Omega} \boldsymbol{f}  \cdot \nabla_T \psi \, dS .
\end{equation}

Define the null space $K_{T, \Bsigma}^p(\Omega)$:
\begin{equation}\nonumber
K_{T, \Bsigma}^p(\Omega) = \{  \Bu \in L^p(\Omega) : \ {\rm div}~\Bu = 0, \ {\rm curl}~(\Bsigma \Bu ) = 0, \ \mbox{in}\ \Omega,\ \mbox{and}\ \Bu \cdot \Bn = 0 \ \mbox{on}\ \partial \Omega                \}.
\end{equation}
In particular, denote
\begin{equation}\nonumber
K_{T, \Bsigma}(\Omega) = \{ \Bu \in C^\infty(\overline{\Omega}) : \ {\rm div}~\Bu = 0, \ {\rm curl}~(\Bsigma \Bu ) = 0, \ \mbox{in}\ \Omega,\ \mbox{and}\ \Bu \cdot \Bn = 0 \ \mbox{on}\ \partial \Omega         \}.
\end{equation}
The functions in $K_{T, \Bsigma}$ are called $\Bsigma$-harmonic fields of the magnetic type.

Define the null space $K_{N, \Bepsilon}^p(\Omega)$:
\begin{equation}\nonumber
K_{N, \Bepsilon}^p(\Omega) = \{  \Bu \in L^p(\Omega) : \ {\rm div}~(\Bepsilon \Bu) = 0, \ {\rm curl}~\Bu  = 0, \ \mbox{in}\ \Omega,\ \mbox{and}\ \Bu \times \Bn = 0 \ \mbox{on}\ \partial \Omega \}.
\end{equation}
In particular, denote
\begin{equation}\nonumber
K_{N, \Bepsilon}(\Omega) = \{ \Bu \in C^\infty(\overline{\Omega}) : \ {\rm div}~(\Bepsilon \Bu) = 0, \ {\rm curl}~ \Bu  = 0, \ \mbox{in}\ \Omega,\ \mbox{and}\ \Bu \times \Bn = 0 \ \mbox{on}\ \partial \Omega         \}.
\end{equation}
The functions in $K_{N, \Bepsilon}$ are called $\Bepsilon$-harmonic fields of the electric type.

In particular, when $\Bsigma = \Bepsilon = Id$, denote
\begin{equation}\nonumber
K_T^p(\Omega) = K_{T, Id}^p(\Omega), \ \ \ \ K_T(\Omega)= K_{T, Id}(\Omega),
\end{equation}
\begin{equation}\nonumber
 K_N^p(\Omega)= K_{N, Id}^p(\Omega),\  \ \ K_N(\Omega) = K_{N, Id}(\Omega).
\end{equation}

%%%%%%%%%%%%%%%%%%%%Main Results%%%%%%%%%%%%%%%%%%%%%%%%%%%%%%%%%%

\vspace{3mm}
Next, let us state our main results. Regarding the problem \eqref{1-1}, we have
\begin{theorem}\label{theorem1}
Let $m \in \mathbb{N}^*$, $1< p < + \infty$. Assume that $\BJ , \rho \in W^{m-1, p}(\Omega)$, $\lambda \in W^{m-\frac1p, p}(\partial \Omega)$. and $(\BJ, \rho, \lambda)$ satisfy the following compatibility conditions,
\begin{equation}\label{theorem1-1}
{\rm div}~\BJ = 0,\ \ \mbox{in}\ \Omega, \ \ \ <\BJ\cdot \Bn, \, 1>_{\Gamma_i} = 0,\ \ 0\leq i \leq N_1, \ \ \ \int_{\Omega} \rho \, dx = \int_{\partial \Omega} \lambda \, dS.
\end{equation}
Then there exists one solution $\Bu_0 \in W^{m, p}(\Omega)$ to the problem \eqref{1-1}, with the estimate
\begin{equation}\label{theorem1-2}
\|\Bu_0\|_{W^{m, p}(\Omega)} \leq C \left(  \|\BJ \|_{W^{m-1, p}(\Omega)} + \|\rho\|_{W^{m-1, p}(\Omega)} + \|\lambda \|_{W^{m- \frac1p, p}(\partial \Omega)} \right).
\end{equation}
Moreover, for every solution $\Bu \in L^p(\Omega)$ to the problem \eqref{1-1}, $\Bu$ can be represented as follows:
\begin{equation}\nonumber 
\Bu = \Bu_0 +  \Bh,\ \ \ \ 
\Bh\in K_{T, \Bsigma}^p(\Omega).
\end{equation}
\end{theorem}

\begin{remark}
Let us note that the compatibility conditions \eqref{theorem1-1} are sufficient and necessary. The necessarity of 
\begin{equation}\nonumber
{\rm div}~\BJ = 0,\ \ \ \mbox{in}\ \Omega, \ \ \ \ \ \mbox{and}\ \ \ \ \int_{\Omega} \rho \, dx = \int_{\partial \Omega}
\lambda\, dS
\end{equation}
is obvious. For $1\leq i \leq N_1$, let $\mu_i$ be a function in $C^\infty(\overline{\Omega})$, which is equal to $1$ in a neighbourhood of $\Gamma_i$ and vanishes in a neighbourhood of $\Gamma_k$, $k \neq i$, then 
\begin{equation}\nonumber
< \BJ \cdot \Bn, \, 1>_{\Gamma_i} = <{\rm curl}~(\mu_i \Bu_0) \cdot \Bn , \, 1>_{\partial \Omega} =0. 
\end{equation}
\end{remark}

\vspace{3mm}

As a corollary of Theorem \ref{theorem1}, we derive the following inequality of Friedrichs type.
\begin{theorem}\label{theorem2}
Let $m \in \mathbb{N}^*$, $1< p < + \infty$. Suppose that $\Bu \in L^p (\Omega)$ satisfies  ${\rm curl}~(\Bsigma \Bu ) \in W^{m-1, p}(\Omega)$, ${\rm div}~\Bu \in W^{m-1, p}(\Omega)$, and $\Bu \cdot \Bn \in W^{m - \frac1p, p}(\partial \Omega)$. Then we have $\Bu \in W^{m, p}(\Omega)$ with
\begin{equation}\label{theorem2-1}
\|\Bu\|_{W^{m, p}(\Omega)} \leq C \left( \|\Bu\|_{L^p(\Omega)} + \|{\rm curl}~(\Bsigma \Bu) \|_{W^{m-1, p}(\Omega)} + \|{\rm div}~\Bu\|_{W^{m-1, p}(\Omega)} + \| \Bu \cdot \Bn \|_{W^{m- \frac1p, p}(\partial \Omega)} \right).
\end{equation}
\end{theorem}

\begin{remark}
If the domain $\Omega$ is simply connected, i.e., the Betti number $N_2=0$,  \eqref{1-1} admits at most one solution in this case. Then the $W^{m, p}$-estimate for $\Bu$ follows from \eqref{theorem1-2} directly,  and hence the term $\|\Bu\|_{L^p(\Omega)}$ on the right hand of \eqref{theorem2-1} can be omitted.
\end{remark}

\vspace{3mm}

Regarding the system \eqref{1-2}, we have
\begin{theorem}\label{theorem3}
Let $m \in \mathbb{N}^*$, $1< p < +\infty$.  Assume that $\BJ, \rho \in W^{m-1, p}(\Omega)$, $\BLambda \in W^{m - \frac1p, p}(\partial \Omega)$, and 
$(\BJ, \BLambda)$ satisfy the following compatibility conditions, 
\begin{equation}\label{theorem3-compatibility-1} 
{\rm div}~\BJ = 0,\ \ \ \mbox{in}\ \Omega, \ \ \ \ \BLambda \cdot \Bn = 0,\ \ \ \ \ \ \BJ \cdot \Bn = {\rm div}_T  \BLambda,\ \ \mbox{on}\ \partial \Omega, 
\end{equation}
\begin{equation}\label{theorem3-compatibility-2}
\int_\Omega \BJ \cdot \Bvarphi \, dx = - \int_{\partial \Omega} \BLambda \cdot \Bvarphi \, dS,\ \ \ \ \forall \ \Bvarphi \in K_T(\Omega).
\end{equation}
Then there exists one solution $\Bu_0 \in W^{m, p}(\Omega)$ to the problem \eqref{1-2}, with the estimate
\begin{equation}\label{theorem3-2}
\|\Bu_0\|_{W^{m, p}(\Omega)} \leq C \left(  \|\BJ \|_{W^{m-1, p}(\Omega)} + \|\rho\|_{W^{m-1, p}(\Omega)} + \|\BLambda \|_{W^{m- \frac1p, p}(\partial \Omega)} \right).
\end{equation}

Moreover, for every solution $\Bu \in L^p(\Omega)$ to the problem \eqref{1-2}, $\Bu$ can be represented as 
\begin{equation}\nonumber
\Bu = \Bu_0 + \Bh,\ \ \ \ \mbox{with}\ \ \Bh \in K_{N, \Bepsilon}^p(\Omega).
\end{equation}
\end{theorem}

\begin{remark} 
The compatibility conditions \eqref{theorem3-compatibility-1}-\eqref{theorem3-compatibility-2} were proposed by Alonso-Valli\cite{AV}. Let us note that they are sufficient and necessary.  The necessarity of 
\begin{equation}\nonumber
{\rm div}~ \BJ = 0,\ \ \ \mbox{in}\ \Omega,\ \ \ \ \mbox{and}\ \ \ \ \BLambda \cdot \Bn = 0,\ \ \ \mbox{on}\ \partial \Omega
\end{equation} 
is obvious.  And $\BJ \cdot \Bn = {\rm div}_T \BLambda $ on $\partial \Omega$ is implied by the following formula\cite{Monk},
\begin{equation}\nonumber
{\rm curl}~\Bu_0 \cdot \Bn = {\rm div}_T~ ( \Bu_0 \times \Bn) , \ \ \ \mbox{on}\ \partial \Omega.
\end{equation}

Moreover, $\forall \ \Bvarphi \in K_{T}(\Omega)$, 
\begin{equation}\nonumber
\int_{\Omega} {\rm curl}~\Bu_0 \cdot \Bvarphi \, dx = - \int_{\partial \Omega} (\Bu_0 \times \Bn) \cdot \Bvarphi \, dS,  
\end{equation}
which implies that 
\begin{equation}\label{4-3-25}
\int_{\Omega} \BJ \cdot \Bvarphi \, dx = -\int_{\partial \Omega } \BLambda \cdot \Bvarphi \, dS. 
\end{equation}

By the way, as explained in Remark \ref{remark-compatibility-2}, the compatibility condition \eqref{theorem3-compatibility-2} can be replaced by 
\begin{equation}\nonumber
< \BJ \cdot \Bn_j,\, 1>_{\Sigma_j} = \int_{\partial \Sigma_j} (\Bn \times \BLambda ) \cdot \Btau_j \, dl , 
\ \ \ \ \ 1\leq j \leq N_2,
\end{equation}
where $\Bn_j$ is the unit normal on $\Sigma_j$, and $\Btau_j$ is the unit tangential vector of $\partial \Sigma_j$. 
\end{remark}

\vspace{3mm}
Similarly, we get a new inequality of Friedrichs type as a corollary.
\begin{theorem}\label{theorem4} 
Let $m \in \mathbb{N}^*$, $1< p < + \infty$. 
Suppose that $\Bu \in L^p (\Omega)$ with ${\rm curl}~ \Bu  \in W^{m-1, p}(\Omega)$, ${\rm div}~(\Bepsilon \Bu) \in W^{m-1, p}(\Omega)$, and $\Bu \times \Bn \in W^{m - \frac1p, p}(\partial \Omega)$. Then we have $\Bu \in W^{m, p}(\Omega)$ with

\begin{equation}\label{theorem4-1}
\|\Bu\|_{W^{m, p}(\Omega)} \leq C \left( \|\Bu\|_{L^p(\Omega)} + \|{\rm curl }~\Bu \|_{W^{m-1, p}(\Omega)} + \|{\rm div}~(\Bepsilon \Bu)\|_{W^{m-1, p}(\Omega)} + \| \Bu \times \Bn\|_{W^{m- \frac1p, p}(\partial \Omega)} \right).
\end{equation}

\end{theorem}

\begin{remark}
If $\partial \Omega$ has only one part $\Gamma_0$, i.e.,the Betti number $N_1=0$,  \eqref{1-2} admits at most one solution in this case. The $W^{m, p}$-estimate for $\Bu$ follows from \eqref{theorem3-2} directly, and hence the term $\|\Bu\|_{L^p(\Omega)} $ on the right hand of \eqref{theorem4-1} can be omitted. 
\end{remark}

\begin{remark}
In fact, Shen-Song\cite{SS} has derived the $W^{1, p}$-estimate of $\Bu$. However, they did not talk about the solvability and the domain they considered is simply connected.  Here we consider a general case, i.e., $\Omega$ is a multiply connected domain. It requires more analysis of the compatibility conditions and the null space. 
\end{remark}

%%%%%%%%%%%%%%%%%%%%%%%%%%%%%%%%Section 3%%%%%%%%%%%%%%%%%%%%%%%%%%%%%%%%

\section{Proof of Theorems \ref{theorem1}- \ref{theorem2}}

\subsection{Proof of Theorem \ref{theorem1}} In this subsection, we will give the proof of Theorem \ref{theorem1}. 
Let $\Bv = \Bsigma \Bu$. In fact, we consider the following problem for $\Bv$,
\begin{equation}\label{3-1}
\left\{\begin{array}{l}
{\rm curl}~ \Bv = \BJ,\ \ \ \mbox{in}\ \Omega,\\[2mm]
{\rm div}~(\Bsigma^{-1} \Bv )  = \rho,\ \ \ \mbox{in}\ \Omega, \\[2mm]
\Bsigma^{-1} \Bv \cdot \Bn = \lambda, \ \ \ \mbox{on}\ \partial\Omega.
\end{array}
\right.
\end{equation}
The basic idea of the proof is to change \eqref{3-1} into one  standard elliptic equation of divergence form. The first step is to remove $\BJ$. Then the second step is to find one special solution of gradient form for the new problem. To accomplish the first step, let us introduce  one preliminary lemma.

\begin{lemma}\label{lemma3-1} A vector  field  $ \Bz \in L^p({\rm div}; \Omega)$ satisfies
\begin{equation}\nonumber
{\rm div}~\Bz = 0, \ \ \ \mbox{in}\ \ \Omega, \ \ \ \mbox{and}\ \ <\Bz \cdot \Bn, 1>_{\Gamma_i}=0, \ \ 0\leq i \leq N_1,
\end{equation}
if and only if there exists a vector potential $\Bpsi$ in $W^{1, p}(\Omega)$ such that
\begin{equation}\nonumber
\Bz = {\rm curl }~\Bpsi,\ \ \ \mbox{and} \ \ \ {\rm div}~\Bpsi = 0 \ \ \mbox{in}\ \Omega,
\end{equation}
\begin{equation}\nonumber
 \ \ \Bpsi \cdot \Bn = 0,\ \ \mbox{on}\ \partial \Omega, \ \ \ \ < \Bpsi\cdot \Bn, \ 1>_{\Sigma_j} = 0, \ 1\leq j \leq N_2.
\end{equation}
This function $\Bpsi$ is unique and we have the estimate:
\begin{equation}\label{3-5}
\|\Bpsi \|_{W^{1, p}(\Omega)} \leq C \| \Bz \|_{L^p(\Omega)},
\end{equation}
where $C>0$ depends only on $p$ and $\Omega$. Moreover, if $\Bz \in W^{m-1, p}(\Omega)$, $m\geq 2$, we have the estimate
\begin{equation}\label{3-6}
\|\Bpsi \|_{W^{m, p}(\Omega)} \leq C \| \Bz \|_{W^{m-1, p}(\Omega)}.
\end{equation}
\end{lemma}
The existence of the vector potential $\Bpsi$ and regularity estimate \eqref{3-5} is proved in \cite{AS}. It is exactly Theorem 4.1 in \cite{AS}. The high-order estimate \eqref{3-6} is a result
of the inequality \eqref{Sobolev1}.

\begin{proof}[\bf Proof of Theorem \ref{theorem1}]   Due to the compatibility condition \eqref{theorem1-1},
\begin{equation}\nonumber
{\rm div}~\BJ = 0, \ \ \mbox{in}\ \Omega, \ \ \ \ \mbox{and}\ \ \ < \BJ\cdot \Bn, \ 1>_{\Gamma_i} = 0, \ \ 0\leq i \leq N_1.
\end{equation}
It follows from Lemma \ref{lemma3-1} that  there exists a vector potential $\Bpsi \in W^{m, p}(\Omega)$, such that
\begin{equation}\nonumber
\BJ = {\rm curl }~ \Bpsi, \ \ \ {\rm div}~\Bpsi = 0, \ \ \mbox{in}\ \Omega, \ \ \ \mbox{and}\ \ \Bpsi \cdot \Bn = 0, \ \ \mbox{on}\ \partial \Omega.
\end{equation}
Furthermore,
\begin{equation}\label{3-10}
\|\Bpsi \|_{W^{m, p}(\Omega)} \leq C \|\BJ \|_{W^{m-1, p}(\Omega)}.
\end{equation}

Next, we plan to search for one solution of the following system:
\begin{equation}\label{3-11}
\left\{ \begin{array}{l}
{\rm curl }~\Bw = 0, \ \ \ \mbox{in}\ \  \Omega, \\[2mm]
{\rm div}~(\Bsigma^{-1} \Bw )  = \rho - {\rm div}~(\Bsigma^{-1} \Bpsi),\ \ \ \mbox{in}\ \ \Omega,\\[2mm]
\Bsigma^{-1} \Bw \cdot \Bn= \lambda - \Bsigma^{-1} \Bpsi \cdot \Bn,\ \ \ \mbox{on}\ \ \partial \Omega.
\end{array}\right.
\end{equation}
$\Omega$ is a multiply connected domain, ${\rm curl}~\Bw= 0 $ does not imply $\Bw$ is a gradient. However,  here we are searching for  one particular solution which assumed to be a gradient. Consider the
following problem,
\begin{equation}\label{3-12} \left\{ \begin{array}{l}
{\rm div}~(\Bsigma^{-1} \nabla q) = \rho - {\rm div}~(\Bsigma^{-1} \Bpsi), \ \ \ \mbox{in} \ \Omega,\\ [2mm]
\Bsigma^{-1} \nabla q \cdot \Bn = \lambda - \Bsigma^{-1} \Bpsi \cdot \Bn,\ \ \ \mbox{on}\ \partial \Omega.
\end{array}
\right.
\end{equation}
Note that
\begin{equation}\nonumber
\int_{\Omega} \left[ \rho - {\rm div}~(\Bsigma^{-1} \Bpsi )  \right] \, dx = \int_{\partial \Omega} (\lambda - \Bsigma^{-1} \Bpsi \cdot \Bn ) \, dS,
\end{equation}
due to the compatibility condition \eqref{theorem1-1}. Applying the classical theory for elliptic equation with conormal boundary condition\cite{Lieberman},  there exists one unique solution (modulus one constant) $q \in W^{m, p}(\Omega)$,  such that
\begin{equation}\label{3-14} \begin{array}{ll}
\|\nabla q\|_{W^{m, p}(\Omega)} &
\leq C \left(  \left\|\rho - {\rm div}(\Bsigma^{-1}\Bpsi) \right\|_{W^{m-1, p}(\Omega)} + \left\|\lambda - \Bsigma^{-1} \Bpsi \cdot \Bn \right\|_{W^{m -\frac1p, p}(\partial \Omega)}              \right)\\[3mm]
& \leq C \left(    \| \rho \|_{W^{m-1, p}(\Omega)} + \|\lambda \|_{W^{m - \frac1p, p}(\partial \Omega)} + \|\Bpsi \|_{W^{m, p}(\Omega)}             \right)\\[3mm]
& \leq C \left(  \| \rho \|_{W^{m-1, p}(\Omega)} + \|\lambda \|_{W^{m - \frac1p, p}(\partial \Omega)} + \| \BJ \|_{W^{m - 1, p}(\Omega)}          \right),
\end{array}
\end{equation}
where the second inequality is due to trace theorem for Sobolev functions and the last inequality is due to \eqref{3-10}.

Let 
\begin{equation}\nonumber
\Bv_0= \nabla q + \Bpsi,\ \ \ \ \mbox{and}\ \ \ \  \Bu_0 = \Bsigma^{-1} \Bv_0, 
\end{equation}
 it is easy to check that $\Bu_0$ is a solution to \eqref{1-1}. And
\begin{equation}\label{3-15} \begin{array}{ll}
\|\Bu_0\|_{W^{m, p}(\Omega)} & \leq C \|\nabla q + \Bpsi \|_{W^{m, p}(\Omega)}\\[2mm]
& \leq
C \left(  \| \rho \|_{W^{m-1, p}(\Omega)} + \|\lambda \|_{W^{m - \frac1p, p}(\partial \Omega)} + \| \BJ \|_{W^{m - 1, p}(\Omega)}          \right).
\end{array}
\end{equation}
That ends the proof of Theorem \ref{theorem1}.

\end{proof}

\begin{remark}
When $\Omega$ is of class $C^{m+1}$, and $\Bsigma \in C^m(\overline{\Omega})$, the result in  Theorem \ref{theorem1} also holds. 
\end{remark}

%%%%%%%%%%%%%%%%%%%%%%%%%%%%%%%%%%%%%Null space K_T%%%%%%%%%%%%%%%%%%%%%%%%%%%%%%%%%%%%%%
\subsection{The null space $K_{T, \Bsigma}^p(\Omega)$}

To get some more knowledge about the solutions to \eqref{1-1},  we will study the null space $K_{T, \Bsigma}^p(\Omega)$ in this subsection. One particular basis will be given. The characterization of $K_{T, \Bsigma}^p(\Omega)$ is inspired by that of $K_T^p(\Omega)$ in \cite{ABDG, AS, KY}. Beforehand, let us give some notations which will be used in this particular subsection.  

$\{ \Sigma_j: \ 1\leq j \leq N_2  \}$ is an admissible set of cuts, which cuts $\Omega$ adequately to reduce it to a simply connected domain. Let $\Omega^0=\Omega \setminus  \cup_{j =1}^{N_2} \Sigma_j$  and let us fix a unit normal $\Bn_j$ on each $\Sigma_j$, $1 \leq j \leq N_2$. \\[2mm]
(1)\ For any function $q \in W^{1, 2}(\Omega^0)$, let us denote by $[q]_j$ the jump of $q$ through $\Sigma_j$ (i.e. the differences of the traces of $q$) along $\Bn_j$. \\[2mm]
(2)\ For any distribution $q$ in $\mathcal{D}^{\prime} (\Omega^0)$, let us denote   by $\nabla^0 q$ the gradient of $q$ in $\mathcal{D}^{\prime} (\Omega^0)$, in order to distinguish from the gradient $\nabla q$ in $\mathcal{D}^{\prime}(\Omega)$.\\[2mm]
(3)\ Let us introduce one function space, where  the null space basis comes from,
\begin{equation}
\Theta = \left\{  r \in W^{1, 2}(\Omega^0):\ \ [r]_j = constant,\ 1\leq j \leq N_2            \right\}.
\end{equation}

The following  preliminary lemma gives a characterization of the functions in $\Theta$,  whose proof can be found in \cite{ABDG}.
\begin{lemma}\label{lemma3-2}
Let $r $ belong to $W^{1, 2}(\Omega^0)$. Then $r $ belongs to $\Theta$ if and only if
\begin{equation}\nonumber
{\rm curl }~(\nabla^0 r) = \nabla \times (\nabla^0 r ) = 0, \ \ \ \mbox{ in}\  \Omega.
\end{equation}
\end{lemma}

%%%%%%%%%%%%%%%%%%%%%%%%%basis of K_T^2%%%%%%%%%%%%%%%%%%%%
Now we are ready to give one explicit basis for $K_{T, \Bsigma}^2(\Omega)$.
\begin{theorem}\label{lemma3-3}
The dimension of the null space $K_{T, \Bsigma}^2(\Omega)$ is equal to the Betti number $N_2$. It is spanned by the functions $\{ \Bsigma^{-1} \nabla^0 q_j^T, \ 1\leq j \leq N_2\}$ , where $q_j^T$ is the solution in $W^{1, 2}(\Omega^0)$, unique up to an additive constant, of the problem
\begin{equation}\label{3-21}
\left\{ \begin{array}{l}
- \nabla^0  \cdot ( \Bsigma^{-1} \nabla^0 q_j^T) = 0, \ \ \ \mbox{in}\ \Omega^0,\\[2mm]
\Bsigma^{-1} \nabla^0 q_j^T \cdot \Bn = 0, \ \ \ \mbox{on}\ \partial \Omega,  \\[2mm]
[q_j^T]_k = constant,\ \ \ \mbox{and}\ \ \ \ [\Bsigma^{-1} \nabla^0 q_j^T\cdot \Bn_k ]_k = 0,\ \ \ \ 1\leq k \leq N_2,\\[2mm]
< \Bsigma^{-1} \nabla^0 q_j^T \cdot \Bn_k, \ 1>_{\Sigma_k} = \delta_{jk},\ \ \ 1\leq k \leq N_2.
\end{array}
\right.
\end{equation}
\end{theorem}

\begin{proof}
The proof is decomposed into two parts.  The first part is devoted to the existence of $q_j^T$, while in the second part we verify that $\{ \Bsigma^{-1}\nabla^0 q_j^T;\, 1\leq j \leq N_2 \}$ is a basis
of $K_{T, \Bsigma}^2(\Omega).$

 {\bf Part I}\ \ For $1 \leq j \leq N_2$, let us consider the following variational problem: find $q_j^T$ in $\Theta$, such that for every $\varphi \in \Theta$,
\begin{equation}\label{3-2-1}
\int_{\Omega^0} \Bsigma^{-1} \nabla^0 q_j^T \cdot \nabla^0  \varphi \, dx = [\varphi]_j.
\end{equation}
It follows from  Lax-Milgram theorem that the problem \eqref{3-2-1} has a solution which is unique up to an additive constant. Using \eqref{3-2-1} with $\varphi \in
C_0^\infty(\Omega^0)$,
\begin{equation}\label{3-2-2-1}
\int_{\Omega^0} \Bsigma^{-1} \nabla^0 q_j^T \cdot \nabla^0 \varphi \, dx  = [\varphi ]_{j} = 0,
\end{equation}
hence we obtain that
\begin{equation}\label{3-2-2-2}
- \nabla^0 \cdot \left( \Bsigma^{-1} \nabla^0  q_j^T    \right) = 0 ,\ \ \ \mbox{in}\ \Omega^0.
\end{equation}

%Note that for any function $q \in W^{1, 2}(\Omega^0)$, $\tilde{\nabla } q $ belongs to $L^2(\Omega^0)$, and can be extended to $L^2(\Omega)$. 

Using \eqref{3-2-1} with $\varphi \in C_0^\infty
(\Omega)$,
\begin{equation}\label{3-2-2}
\int_{\Omega^0} \Bsigma^{-1} \nabla^0 q_j^T \cdot \nabla^0 \varphi \, dx = [\varphi]_j =0,
\end{equation}
On the other hand,
\begin{equation}\label{3-2-2-4}
\int_{\Omega^0} \Bsigma^{-1} \nabla^0 q_j^T \cdot \nabla^0 \varphi \, dx = - \int_{\Omega^0} \nabla^0 \cdot (\Bsigma^{-1} \nabla q_j^T ) \cdot \varphi \, dx
+ \sum_{k =1}^{N_2} \int_{\Sigma_k} [ \Bsigma^{-1} \nabla^0 q_j^T \cdot \Bn_k ]_k  \cdot \varphi \, dS.
\end{equation}
Combining the three equalities \eqref{3-2-2-2}-\eqref{3-2-2-4}, we have
\begin{equation}\label{3-2-2-5}
[\Bsigma^{-1} \nabla^0 q_j^T \cdot \Bn_k ]_k = 0,\ \ \ \ \ k = 1, 2, \cdots, N_2,
\end{equation}
due to the arbitrariness of $\varphi$. \eqref{3-2-2-5} means the jump of $\Bsigma^{-1} \nabla^0 q_j^T \cdot \Bn_k$ across any $\Sigma_k$ is zero.

Furthermore, using \eqref{3-2-1} with $\varphi \in C^\infty(\overline{\Omega})$,
\begin{equation} \nonumber
0 = \int_{\Omega^0} \Bsigma^{-1} \nabla^0 q_j^T \cdot \nabla^0 \varphi \, dx =   \int_{\partial \Omega} ( \Bsigma^{-1} \nabla^0 q_j^T \cdot \Bn ) \cdot \varphi \, dS,
\end{equation}
where we used the facts \eqref{3-2-2-2} and \eqref{3-2-2-5}.
It implies that
\begin{equation}\label{3-2-6}
\Bsigma^{-1} \nabla^0  q_j^T  \cdot \Bn = 0, \ \ \ \mbox{on}\ \partial \Omega.
\end{equation}

Fixing an integer $k$,\ $1 \leq k \leq N_2$,  choosing $r$ in $\Theta$, such that $[r]_j =  \delta_{jk}$ for any $j$, and applying \eqref{3-2-1} with $\varphi = r$,
\begin{equation}\label{3-2-7}
< \Bsigma^{-1} \nabla^0 q_j^T \cdot \Bn_k,  1>_{\Sigma_k} = \int_{\Omega^0} \Bsigma^{-1} \nabla^0 q_j^T \cdot \nabla^0 r \, dx
= \delta_{jk}.
\end{equation}
Hence, the solution to the variational problem \eqref{3-2-1} is also the solution to \eqref{3-21}. On the other hand, it is easy to check that every solution of \eqref{3-21} also solves \eqref{3-2-1}. Thus \eqref{3-21} admits one unique solution(modulus one constant) $q_j^T \in 
\Theta$ ( up to an additive constant).

{\bf Part II}\ \  It follows from the proof in {\bf Part I} and Lemma \ref{lemma3-2}  that 
\begin{equation}\nonumber
{\rm curl}~(\nabla^0 q_j^T ) = 0, \ \ \ \mbox{in}\ \Omega, \ \ \ \ \ \Bsigma^{-1} \nabla^0 q_j^T  \cdot \Bn = 0,
\ \ \ \mbox{on}\ \partial \Omega,
\end{equation}
and 
\begin{equation}\nonumber
- \nabla^0 \cdot \left(\Bsigma^{-1} \nabla^0 q_j^T   \right) =  0, \ \ \ \mbox{in}\ \Omega^0.
\end{equation}
We will prove that in fact
\begin{equation}\label{additional}
- {\rm div}~ \left(\Bsigma^{-1} \nabla^0 q_j^T   \right) =  0, \ \ \ \mbox{in}\ \Omega.
\end{equation}
Note that for every function $\varphi \in C_0^\infty(\Omega)$, 
\begin{equation}\nonumber
\nabla \varphi = \nabla^0 \varphi ,\ \ \ \ \mbox{in}\ \Omega^0. 
\end{equation}
Hence, it holds that 
\begin{equation}\nonumber
\int_{\Omega} \Bsigma^{-1} \nabla^0 q_j^T \cdot \nabla \varphi \, dx 
= \int_{ \Omega^0} \Bsigma^{-1} \nabla^0 q_j^T \cdot \nabla^0 \varphi \, dx= 0 ,
\end{equation}
which implies \eqref{additional}. Hence, $\Bsigma^{-1} \nabla^0   q_j^T \in K_{T, \Bsigma}^2(\Omega)$. 

Due to the last equality in \eqref{3-21}, the functions $\left\{ \Bsigma^{-1} \nabla^0 q_j^T , \ 1 \leq j \leq N_2 \right\}$, are linearly independent. Next, we show that they span the space $K_{T, \Bsigma}^2(\Omega)$.
Let $\Bu$ be any function in $K_{T, \Bsigma}^2(\Omega)$ and consider the function
\begin{equation}\nonumber
\Bw =   \Bu - \sum_{j=1}^{N_2} < \Bu \cdot \Bn_j,\  1>_{\Sigma_j} \Bsigma^{-1} \nabla^0 q_j^T.
\end{equation}
It is easy to check that 
\begin{equation}\nonumber
\nabla^0 \times (\Bsigma \Bw) = 0,\ \ \ \mbox{in}\ \Omega^0.
\end{equation}
Since $\Omega^0$ is simply connected,  there exists a function $q \in W^{1,2}(\Omega^0)$, such that
\begin{equation}\nonumber
\Bsigma \Bw = \nabla^0 q,\ \ \ \ \mbox{ in}\  \Omega^0.
\end{equation}

¡¡
On the other hand, $q_j^T \in \Theta$, it follows that 
\begin{equation}\nonumber
{\rm curl }~ (\nabla^0 q_j^T) = 0,\ \ \ \ \mbox{in} \ \Omega.
\end{equation}
Consequently, 
\begin{equation}\nonumber
{\rm curl}~(\nabla^0 q) = {\rm curl}~(\Bsigma \Bu) - \sum_{j =1}^{N_2} < \Bu\cdot \Bn_j, 1>_{\Sigma_j} 
 {\rm curl}~\left( \nabla^0  q_j^T \right) = 0,\ \ \ \ \mbox{in}\ \Omega.
\end{equation}
According to Lemma \ref{lemma3-2}, $q \in \Theta$. 

Hence, 
\begin{equation}\label{3-2-15} 
\begin{array}{l}
\displaystyle \int_{\Omega^0} \Bsigma  \Bw \cdot \Bw \, dx \\[3mm] = \displaystyle  \int_{\Omega^0} \Bw \cdot \nabla^0 q \, dx \\[3mm]
=  \displaystyle \int_{\Omega^0} \Bu \cdot \nabla^0 q\, dx - \sum_{j =1}^{N_2} < \Bu \cdot \Bn_j, 1>_{\Sigma_j} \cdot
\int_{\Omega^0} \Bsigma^{-1} \nabla^0  q_j^T \cdot \nabla^0 q\, dx \\[3mm]
= \displaystyle \sum_{j =1}^{N_2} < \Bu \cdot \Bn_j, 1>_{\Sigma_j} \cdot [q]_j - \sum_{j =1}^{N_2} < \Bu \cdot \Bn_j, 1>_{\Sigma_j} \cdot [q]_j\\[4mm]
=0,
\end{array}
\end{equation}
where the third equality is due to integration by parts and \eqref{3-2-1}. It implies that $\Bw$ is zero, i. e., 
\begin{equation}\nonumber \displaystyle
\Bu = \sum_{j =1}^{N_2} < \Bu \cdot \Bn_j, 1>_{\Sigma_j} \Bsigma^{-1} \nabla^0 q_j^T,
\end{equation}
which says that $\{ \Bsigma^{-1} \nabla^0 q_j^T,\ 1\leq j \leq N_2 \}$ span the space $K_{T, \Bsigma}^2(\Omega)$. 
That ends the proof of Theorem \ref{lemma3-3}.
\end{proof}

\vspace{3mm}

%%%%%%%%%%%%%%%%%%%%%%%%%%%%%%%%%%%%%Indentity between K_T^p and K_T

Next theorem will prove the identity between $K_{T, \Bsigma}^p(\Omega)$and $K_{T, \Bsigma}(\Omega)$. The proof is based on the fact that $K_T^p(\Omega) = K_T^2(\Omega) = K_T(\Omega)$, which has been proved in
\cite{ KY}.

\begin{theorem}\label{theorem3-3}
For every $1 < p < \infty$, $K_{T, \Bsigma}^p(\Omega) = K_{T, \Bsigma}(\Omega)$.
\end{theorem}

\begin{proof}
Obviously, $K_{T, \Bsigma}(\Omega) \subseteq K_{T, \Bsigma}^p(\Omega)$. It suffices to  prove that for every function $\Bu$ belonging to $K_{T, \Bsigma}^p(\Omega)$, $\Bu \in C^\infty(\overline{\Omega})$.

Suppose that $\Bu \in K_{T, \Bsigma}^p(\Omega)$, let $\Bv = \Bsigma \Bu$. According to Theorem 2.1  in \cite{KY}, $\Bv$ has the following decomposition,
\begin{equation}\nonumber
\Bv = \Bh + {\rm curl }~\Bw + \nabla q,
\end{equation}
where $ \Bh \in K_T^p(\Omega)= K_T(\Omega)$, $q \in W^{1, p}(\Omega)$, and $\Bw \in W^{1, p}(\Omega)$, with
\begin{equation}\nonumber
{\rm div}~\Bw = 0, \  \ \mbox{in} \ \Omega, \ \ \ \Bw\times \Bn= 0, \ \ \mbox{on}\  \partial \Omega.
\end{equation}
And $\Bh$ is unique, $q$ is unique up to an additive constant and $\Bw $ is unique up to an additive element of $K_N^p(\Omega)$.

In fact, $q$ is the solution
to the following system,
\begin{equation} \nonumber \left\{
\begin{array}{l}
\Delta q = {\rm div}~\Bv ,\ \ \ \mbox{in}\ \Omega, \\[2mm]
\nabla q \cdot \Bn = \Bv \cdot \Bn,\ \ \ \mbox{on}\ \partial \Omega.
\end{array}
\right. \end{equation}
Hence,
\begin{equation}\nonumber
{\rm div}~( \Bv  - \nabla q ) = 0,\ \ \mbox{in}\ \Omega, \ \ \ \mbox{and}\ \ (\Bv - \nabla q)\cdot \Bn = 0,\ \ \mbox{on}\ \partial \Omega.
\end{equation}
Since ${\rm curl }~\Bv = 0$ in $\Omega$, it can be verified that
\begin{equation}\nonumber
{\rm curl}~( \Bv  - \nabla q ) = 0,\ \ \ \mbox{in} \ \Omega.
\end{equation}
Therefore, $\Bv - \nabla q \in K_T^p(\Omega)$. Due to the uniqueness of decomposition, ${\rm curl  }~\Bw = 0$, i.e., $\Bv = \Bh + \nabla q$.

Next we will discuss the regularity of $q$. Since
\begin{equation} \nonumber
{\rm div}~(\Bsigma^{-1} \Bv)= 0\ \ \mbox{in}\ \Omega, \ \ \ \mbox{and}\ \ \Bsigma^{-1} \Bv\cdot \Bn = 0\ \ \mbox{on}\ \partial \Omega,
\end{equation}
 $q$ is the solution of the following system,
\begin{equation}\nonumber
\left\{\begin{array}{l}
{\rm div}~(\Bsigma^{-1} \nabla q ) = - {\rm div}~(\Bsigma^{-1} \Bh), \ \ \ \mbox{in}\ \Omega,\\[2mm]
\Bsigma^{-1}\nabla q \cdot \Bn = - \Bsigma^{-1} \Bh \cdot \Bn,\ \ \ \mbox{on}\ \partial \Omega.
\end{array}
\right.
\end{equation}
Since $\Bh \in C^\infty(\overline{\Omega})$, it follows from the classcial regularity theory for elliptic equation with conormal boundary condition \cite{Lieberman} that $q \in C^\infty(\overline{\Omega})$. Consequently, $\Bv\in C^\infty(\overline{\Omega})$ and $\Bu \in C^\infty(\overline{\Omega})$.  That ends the proof of Theorem \ref{theorem3-3}.

\end{proof}

\begin{remark}
If the domain $\Omega$ is of class $C^{2}$, it was proved in \cite{AS} that $K_T^p(\Omega) = K_T^2(\Omega)$. Making use of this fact and following the same line as above, we can also prove that
\begin{equation}\nonumber
K_{T, \Bsigma}^p(\Omega) =  K_{T, \Bsigma}^2(\Omega),\ \ \ \ 1< p< \infty.
\end{equation}
\end{remark}

%%%%%%%%%%%%%%%%%%%%%%%%%%%%%Subsection Friedrich's inequality%%%%%%%%%%%%%%%%%%%%%%%%%

\subsection{Proof of Theorem \ref{theorem2}}

In this subsection, we will prove Friedrichs inequality \eqref{theorem2-1} involving normal boundary value. As remarked in Section 1, when $\Bsigma $ is identity matrix, \eqref{theorem2-1} has been proved in different function spaces and different domains. For the general case, we can not follow the classical method by setting up one integral equality connecting $\nabla \Bu$, ${\rm div}~\Bu$, ${\rm curl}~(\Bsigma \Bu)$ and $\Bu\cdot \Bn$ easily, which is the method applied in \cite{AS}. Our proof lies on the solvability of \eqref{1-1} and the characterization of the null space $K_{T, \Bsigma}^p(\Omega)$.

\begin{proof}[\bf Proof of Theorem \ref{theorem2}] Let $\BJ = {\rm curl}~(\Bsigma \Bu)$, $\rho = {\rm div}~\Bu$, and $\lambda = \Bu \cdot \Bn$. According to Theorem \ref{theorem1}, there exists a function $\Bu_0 \in W^{m, p}(\Omega)$, such that
\begin{equation}\label{3-1-1}
\left\{ \begin{array}{l}
{\rm curl }~(\Bsigma \Bu_0) = J, \ \ \ \mbox{in}\ \Omega,\\[2mm]
{\rm div}~\Bu_0 = \rho, \ \ \ \mbox{in}\ \Omega,\\[2mm]
\Bu_0 \cdot \Bn = \lambda,\ \ \ \mbox{on}\ \partial \Omega,
\end{array}
\right.
\end{equation}
with the estimate
\begin{equation}\label{3-3-2}
\|\Bu_0\|_{W^{m, p}(\Omega)} \leq C \left( \| \BJ \|_{W^{m-1, p}(\Omega)} + \|\rho \|_{W^{m-1, p}(\Omega)} + \|\lambda \|_{W^{m-\frac1p, p}(\partial \Omega)}          \right).
\end{equation}
Let $\Bv =  \Bu  - \Bu_0$. Then $\Bv \in K_{T, \Bsigma}^p(\Omega)$. As proved in Theorem \ref{lemma3-3}, 
\begin{equation}\label{3-3-3}
\Bv = \sum_{j = 1}^{N_2} < \Bv \cdot \Bn_j,  1>_{\Sigma_j}  \Bsigma^{-1} \nabla^0 q_j^T,
\end{equation}
where $ \{ \sigma^{-1} \nabla^0 q_j^T;\, 1\leq j \leq N_2 \} $ is the basis of $K_{T, \Bsigma}^2(\Omega)$, derived in Subsection 3.2. It follows from Theorem \ref{theorem3-3} that,
$$\Bsigma^{-1} \nabla^0 q_j^T \in K_{T, \Bsigma}^2(\Omega) = K_{T, \Bsigma} (\Omega), $$
which says that $\Bsigma^{-1} \nabla^0 q_j^T \in C^\infty(\overline{\Omega})$.  Let $C_{j} = \| \Bsigma^{-1} \nabla^0 q_j^T\|_{W^{m, p}(\Omega)}< + \infty$,
\begin{equation}\label{3-3-4} \begin{array}{ll}
\| \Bv\|_{W^{m, p}(\Omega)} & \displaystyle \leq C \sum_{j=1}^{N_2} \| \Bv\cdot \Bn_j \|_{W^{-\frac1p, p}(\Sigma_j )} \cdot \| \Bsigma^{-1} \nabla^0  q_j^T\|_{W^{m, p}(\Omega)} \\[4mm]
& \displaystyle \leq  C \sum_{j =1}^{N_2} \left( \| \Bu - \Bu_0\|_{L^p(\Omega)} + \|  {\rm div}~( \Bu - \Bu_0)         \|_{L^p(\Omega)}              \right) \cdot C_{j}\\[4mm]
& \leq C \left(  \| \Bu\|_{L^p(\Omega)} + \| \Bu_0\|_{L^p(\Omega)}          \right),
\end{array}
\end{equation}
where the second inequality is due to \eqref{normal}.

Consequently, $\Bu \in W^{m, p}(\Omega)$ and
\begin{equation}\label{3-3-5} \begin{array}{ll}
\| \Bu \|_{W^{m, p}(\Omega)}&  \leq C \left( \| \Bu\|_{L^p(\Omega)} + \|\BJ \|_{W^{m-1, p}(\Omega)} + \|\rho\|_{W^{m-1, p}(\Omega)} + \| \lambda \|_{W^{m -\frac1p, p}(\partial \Omega)} \right)\\[3mm]
& \leq C \left(  \| \Bu\|_{L^p(\Omega)} + \|{\rm curl}~(\Bsigma \Bu)\|_{W^{m-1, p}(\Omega)} + \| {\rm div}~\Bu \|_{W^{m-1, p}(\Omega)} + \| \Bu \cdot \Bn\|_{W^{m -\frac1p, p}(\partial \Omega)}               \right).
\end{array}
\end{equation}
That ends the proof of Theorem \ref{theorem2}. 

\end{proof}

\begin{remark}\label{remark3-3}
According to the above proof, we can have another estimate for $\| \Bv\|_{W^{m,p}(\Omega)}$,
\begin{equation}\nonumber
\|\Bv\|_{W^{m, p}(\Omega)} \leq C \sum_{j =1}^{N_2}  \left| < \Bv \cdot \Bn_j, \ 1 >_{\Sigma_j} \right| 
\cdot \| \Bsigma^{-1}  \nabla^0 q_j^T\|_{W^{m, p}(\Omega)} \leq C \sum_{j =1}^{N_2} \left| < \Bv \cdot \Bn_j, \ 1 >_{\Sigma_j} \right| .
\end{equation}
Consequently, the $W^{m, p}$-estimate for $\Bu$ can also be written in another form, 
\begin{equation}\label{3-3-5} \begin{array}{l}
\| \Bu \|_{W^{m, p}(\Omega)} \\[2mm] \displaystyle
\leq  C \sum_{j =1}^{N_2}  \left| < \Bu \cdot \Bn_j, \ 1 >_{\Sigma_j} \right|  + 
C \sum_{j =1}^{N_2}  \left| < \Bu_0 \cdot \Bn_j, \ 1 >_{\Sigma_j} \right|  + \| \Bu_0\|_{W^{m, p}(\Omega)} \\ [2mm]
 \leq  \displaystyle C  \sum_{ j =1}^{N_2}   \left| < \Bu \cdot \Bn_j, \ 1 >_{\Sigma_j} \right|    + C\left( \|{\rm curl}~(\Bsigma \Bu)\|_{W^{m-1, p}(\Omega)} + \| {\rm div}~\Bu \|_{W^{m-1, p}(\Omega)} + \|\Bu \cdot \Bn\|_{W^{m -\frac1p, p}(\partial \Omega)}               \right).
\end{array}
\end{equation}

\end{remark}

%%%%%%%%%%%%%%%%%%%%%%%%%%%%%Section 4%%%%%%%%%%%%%%%%%%%%%%%%%%%%%%%
\section{Proof of Theorems \ref{theorem3}- \ref{theorem4}}

\subsection{Standard ${\rm div}-{\rm curl}$ system} 

In this section, we will give the proof of  Theorem \ref{theorem3}. Our basic idea is also to change the system \eqref{1-2} into one standard elliptic equation of divergence form ${\rm div}~( \Bepsilon \nabla q) = f$. The first step is to remove $\BJ$ and $\BLambda$ on the right hand of \eqref{1-2}. Once we remove $\BJ$ and $\BLambda$, the solution desired is one curl free vector with zero tangential boundary value, which is a gradient.  To accomplish the first step, we consider the following standard ${\rm div}$-${\rm curl}$ system in this subsection, 
\begin{equation}\label{4-1-1}
\left\{
\begin{array}{l}
{\rm curl }~\Bv = \BJ, \ \ \ \mbox{in}\ \Omega,\\[2mm]
{\rm div}~\Bv = \rho , \ \ \ \mbox{in} \ \Omega, \\[2mm]
\Bv \times \Bn = \BLambda,\ \ \mbox{on} \ \partial \Omega.
\end{array}
\right.
\end{equation}
The solvability of \eqref{4-1-1} in $H^m$-spaces has been derived in \cite{AV}. However, here we require the solvability in 
$W^{m, p}$-spaces. 
When $\rho =0$ and $\BLambda =0$, the solvability in $W^{m, p}$-spaces has been derived in \cite{AS, KY}.

Define
\begin{equation}\nonumber
X_T^p(\Omega) = \left\{ \Bv \in L^p(\Omega):\ \ {\rm div}~\Bv \in L^p(\Omega),\ \ \ {\rm curl}~\Bv\in L^p(\Omega),\ \ \ \Bv\cdot \Bn =0,\ \ \mbox{on}\ \partial \Omega  \right\},
\end{equation}
with the norm 
\begin{equation}\nonumber
\| \Bv\|_{X_T^p(\Omega) } = \|{\rm curl}~\Bv\|_{L^p(\Omega)} + \|{\rm div}~\Bv\|_{L^p(\Omega)} + \| \Bv\|_{L^p(\Omega)}.
\end{equation}

Define
\begin{equation} \nonumber
V_T^p(\Omega) = \left\{  \Bv \in X_T^p(\Omega): \ \ {\rm div}~ \Bv = 0,\ \ \mbox{in}\ \Omega, \ \ \ < \Bv \cdot \Bn_j, \, 1>_{\Sigma_j}=0,\ 1\leq j \leq N_2 \right\}.
\end{equation}

The first lemma is a generalized version of Lax-Milgram theorem, and the second lemma gives a Inf-Sup condition. Their proof can be found in \cite{AS}. 

\begin{lemma}\label{lemma4-0-1}
Let $X$ and $M$ be two reflexive Banach spaces, and $X^{\prime}$ and $M^{\prime}$ be their dual spaces. Let $a$ be the continuous 
bilinear form defined on $X \times M$, let $A \in \mathcal{L}(X; M^{\prime}) $ and $A^{\prime} \in \mathcal{L}(M; X^{\prime})$ be 
the operators defined by 
\begin{equation}\nonumber
\forall \ v \in X,\ \ \ \forall \ w \in M,\ \ \ \ a(v, w) = <Av,\, w> = <v,\, A^{\prime}w >
\end{equation}
and $V= Ker\, A$. The following statements are equivalent: 
\begin{enumerate}
  \item There exists $\beta > 0$ such that 
  \begin{equation}\nonumber
\inf_{\begin{array}{l} w \in M \\ w\neq 0 \end{array}} \sup_{ \begin{array}{l}  v \in X \\
v \neq 0 \end{array}}  \frac{a (v, w) }{\|v\|_{X} \cdot \|w\|_{M}} \geq \beta. 
 \end{equation}

  \item The operator $A : \frac{X}{V} \mapsto M^{\prime}$ is an isomorphism and $\frac{1}{\beta}$ is the 
  continuity constant of $A^{-1}$. 

\item The operator $A^{\prime}: M \mapsto X^{\prime} \perp V $ is an isomorphism and $\frac{1}{\beta}$ is the continuity constant 
 of $(A^{\prime})^{-1}$. 

  \end{enumerate}

\end{lemma}

\begin{lemma}\label{lemma4-0-2} Let $1 <  p < +\infty$. 
The following Inf-Sup condition holds: there exists a constant $\beta > 0$, such that 

\begin{equation}\label{4-0-1}
\inf_{\begin{array}{l} \Bvarphi \in V_{T}^{p^{\prime}}(\Omega) \\ \Bvarphi \neq 0 \end{array}}
\sup_{\begin{array}{l}   \Bxi \in V_T^p(\Omega) \\ \Bxi \neq 0 \end{array}} 
\frac{\int_{\Omega} {\rm curl}~\Bxi \cdot {\rm curl }~\Bvarphi \, dx   }{\|\Bxi\|_{X_T^p(\Omega)}  \cdot \|\Bvarphi\|_{X_T^{p^{\prime}}(\Omega)}   } \geq \beta. 
\end{equation}
\end{lemma}

\begin{lemma}\label{lemma4-0-2-1} Let $1 <  p < +\infty $. For every function $\Bv \in X_T^p(\Omega)$, it holds that 
\begin{equation}\label{4-0-4}
\|\Bv\|_{W^{1, p}(\Omega) } \leq C \| \Bv \|_{X_T^p(\Omega)} + C \sum_{j =1}^{N_2} 
\left| < \Bv \cdot \Bn_j, \, 1>_{\Sigma_j}    \right|.
\end{equation}
\end{lemma}
The proof of Lemma \ref{lemma4-0-2-1} is implied in \cite{AS, KY}. It also can be found in Remark \ref{remark3-3}.

\begin{lemma}\label{lemma4-0-3} Let $ 1 < p < +\infty$. 
Assume that $\BJ \in L^p(\Omega)$, $\BLambda \in W^{1- \frac{1}{p}, p }(\partial \Omega)$, and $(\BJ, \BLambda)$ satisfy the following 
compatibility conditions
\begin{equation}\label{compatibility-4-1}
{\rm div}~\BJ = 0, \ \ \mbox{in}\ \Omega, \ \ \ \ \BLambda \cdot \Bn = 0,\ \ \ \ \BJ \cdot \Bn = {\rm div}_T ~\BLambda,\ \ \mbox{on}\ \partial \Omega,
\end{equation}
and 
\begin{equation}\label{compatibility-4-2}
\int_{\Omega} \BJ \cdot \Bvarphi \, dx = - \int_{\partial \Omega} \BLambda \cdot \Bvarphi \, dS,\ \ \ \ \forall \ \Bvarphi \in K_T(\Omega). 
\end{equation}
Then the following problem 
\begin{equation}\label{4-0-5}
\left\{ \begin{array}{l}
- \Delta \Bxi = \BJ,\ \ \ \ \mbox{in}\ \Omega, \\[2mm]
{\rm div}~ \Bxi = 0,\ \ \ \ \mbox{in}\ \Omega, \\[2mm]
\Bxi \cdot \Bn = 0,\ \ \ \ \mbox{on}\ \partial \Omega,\\[2mm]
{\rm curl }~\Bxi \times \Bn = \BLambda, \ \ \ \ \mbox{on}\ \partial \Omega,\\[2mm]
< \Bxi \cdot \Bn_j, \, 1>_{\Sigma_j} = 0,\ \ \ \ 1\leq j \leq N_2
\end{array}
\right.
\end{equation}
has a unique solution $\Bxi \in W^{2, p}(\Omega)$, and we have the estimate 
\begin{equation}\label{4-0-6}
\|\Bxi \|_{W^{2, p}(\Omega)} 
\leq C \left(  \|\BJ \|_{L^p(\Omega)} + \|\BLambda \|_{W^{1 - \frac1p, p }(\partial \Omega) } \right). 
\end{equation}
\end{lemma}

\begin{proof}
First, let us consider the following problem: \, find $\Bxi \in V_T^p(\Omega)$ such that 
\begin{equation}\label{4-0-9}
\forall\ \Bphi \in V_T^{p^{\prime}} (\Omega),\ \ \ \ \ \int_{\Omega} {\rm curl}~\Bxi \cdot {\rm curl}~\Bphi \, dx
= \int_{\Omega} \BJ \cdot \Bphi \, dx + \int_{\partial \Omega} \Lambda \cdot \Bphi \, dS. 
\end{equation}

According to Lemma \ref{lemma4-0-2}, the left hand of \eqref{4-0-9} satisfies the Inf-Sup condition. On the other hand, 
\begin{equation}\nonumber \begin{array}{l}
\left| \int_{\Omega} \BJ \cdot \Bphi \, dx + \int_{\partial \Omega} \BLambda \cdot \Bphi \, dS   \right| \\[3mm]
\leq C 
\| \BJ \|_{L^p(\Omega)} \cdot \| \Bphi \|_{W^{1, p^{\prime}}(\Omega)}  + C \|\BLambda \|_{W^{1 - \frac1p, p} (\partial \Omega)} 
\cdot \|\Bphi \|_{W^{1, p^{\prime}}(\Omega)} \\[3mm]
\leq C \left( \|\BJ \|_{L^p(\Omega)} + \|\BLambda\|_{W^{1 - \frac1p, p} (\partial \Omega)} \right) \cdot \|\Bphi \|_{X_T^{p^\prime}(\Omega)},
\end{array}
\end{equation}
due to Lemma \ref{lemma4-0-2-1}. 
By virtue of Lemma \ref{lemma4-0-1}, the problem \eqref{4-0-9} has a unique solution $\Bxi \in V_T^p(\Omega)$, with the estimate
\begin{equation}\label{4-0-11}
\|\Bxi \|_{ X_T^p (\Omega)} \leq C \left( \| \BJ \|_{L^p(\Omega) } + \| \BLambda\|_{W^{ 1 -\frac1p, p} (\partial \Omega) }  \right).
\end{equation} 
Since $\Bxi \in V_T^p(\Omega)$, $< \Bxi \cdot \Bn_j, \,  1>_{\Sigma_j}= 0$, $1 \leq j \leq N_2$. According to Lemma \ref{lemma4-0-2-1},  
\begin{equation}\label{4-0-11-1}
\|\Bxi \|_{W^{1, p}(\Omega)} \leq C \| \Bxi\|_{X_T^p(\Omega)} \leq C\left( \|\BJ \|_{L^p(\Omega) } + \|\BLambda\|_{W^{ 1 -\frac1p, p} (\partial \Omega) }  \right).
\end{equation}

Next, we extend \eqref{4-0-9} to any test function $\Bphi$ in $X_T^{p^{\prime}}(\Omega) $. Consider the unique solution(up to an additive constant) $\chi \in 
W^{1, p^{\prime}}(\Omega)$ of the Neumann problem
\begin{equation}\nonumber
\left\{ \begin{array}{l}
\Delta \chi = {\rm div}~\Bphi,\ \ \ \ \mbox{in}\ \Omega,\\[2mm]
\frac{\partial \chi}{\partial \Bn} = 0,\ \ \ \ \mbox{on} \ \partial \Omega.
\end{array}
\right.
\end{equation}
Then we set 
\begin{equation}\nonumber
\tilde{\Bphi} = \Bphi - \nabla \chi - \sum_{j =1}^{N_2} \left< (\Bphi - \nabla \chi)\cdot \Bn_j,\, 1 \right>_{\Sigma_j } \nabla^0 q_{j, *}^T,
\end{equation}
where $\left\{  \nabla^0 q_{j, *}^T;\, 1 \leq j \leq N_2 \right\}$ is the basis of $K_T^2(\Omega)$, constructed as in Subsection 3.2 for the case $\Bsigma= Id$. Observe that $\tilde{\Bphi} \in V_T^{p^{\prime}}(\Omega)$, hence
\begin{equation} \label{4-0-12} \begin{array}{ll}
& \displaystyle \int_{\Omega} {\rm curl }~\Bchi \cdot {\rm curl}~\Bphi \, dx \\[3mm]
= & \displaystyle \int_{\Omega} {\rm curl}~\Bchi \cdot {\rm curl}~\tilde{\Bphi}\, dx \\[3mm]
= &\displaystyle \int_{\Omega} \BJ \cdot \tilde{\Bphi} \, dx + \int_{\partial \Omega} \BLambda \cdot \tilde{\Bphi} \, dS\\[3mm]
= &\displaystyle \int_{\Omega} \BJ \cdot \Bphi \, dx + \int_{\partial \Omega} \BLambda \cdot \Bphi \, dS
- \int_{\Omega} \BJ \cdot \nabla \chi \, dx - \int_{\partial \Omega} \BLambda \cdot \nabla \chi \, dS \\[3mm]
& \ \ \displaystyle - \sum_{j =1}^{N_2} \left< ( \Bphi - \nabla \chi) \cdot \Bn_j, \, 1 \right>_{\Sigma_j} 
\left( \int_{\Omega} \BJ \cdot \nabla^0 q_{j, \star}^T \, dx + \int_{\partial \Omega} \BLambda \cdot \nabla^0 q_{j, *}^T \, dS \right).
\end{array}
\end{equation}
Herein, due to the compatibility conditions \eqref{compatibility-4-1}-\eqref{compatibility-4-2} and the fact 
$\nabla^0 q_{j, *}^T \in K_T(\Omega)$, 
\begin{equation}\nonumber
\int_{\Omega} \BJ \cdot \nabla \chi \, dx + \int_{\partial \Omega} \BLambda \cdot \nabla \chi \, dS 
= \int_{\partial \Omega} ( \BJ \cdot \Bn) \cdot \chi  \, dS - \int_{\partial \Omega} {\rm div}_T \BLambda \cdot \chi \, dS = 0,
\end{equation}
and 
\begin{equation}\nonumber
\int_{\Omega} \BJ \cdot \nabla^0 q_{j, *}^T \, dx  + \int_{\partial \Omega} \BLambda \cdot \nabla^0 q_{j, *}^T 
\, dS = 0.
\end{equation}
Hence \eqref{4-0-9} holds for every test function $\Bphi \in X_{T}^{p^{\prime}}(\Omega)$. 

Since $C_0^\infty(\Omega) \subset X_T^{p^{\prime}}(\Omega)$, \eqref{4-0-9} implies that 
\begin{equation}
-\Delta \Bxi = {\rm curl}{\rm curl}~\Bxi - \nabla {\rm div}~\Bxi = {\rm curl}{\rm curl}~ \Bxi = \BJ,\ \ \ \mbox{in}\ \Omega.
\end{equation} 
For every function $\Bvarphi \in C^\infty(\overline{\Omega})$. Choose one $\tilde{\Bvarphi} \in C^\infty(\overline{\Omega})$, such that $\tilde{\Bvarphi} = (\Bvarphi \cdot \Bn) \Bn$ on $\partial \Omega$. Then 
\begin{equation}\nonumber
\Bphi = \Bvarphi - \tilde{\Bvarphi} \in X_T^{p^{\prime}}(\Omega).
\end{equation} 
On one hand, 
\begin{equation}\label{4-0-15} \begin{array}{ll}
\displaystyle \int_{\Omega} {\rm curl}~\Bxi \cdot {\rm curl}~\Bvarphi \, dx & \displaystyle = \int_{\Omega} {\rm curl}{\rm curl}~\Bxi \cdot \Bvarphi \, dx 
+ \int_{\partial \Omega} \left( {\rm curl }~\Bxi \times \Bn\right) \cdot \Bvarphi \, dS \\[3mm]
& \displaystyle 
= \int_{\Omega} \BJ \cdot \Bvarphi \, dx + \int_{\partial \Omega} \left( {\rm curl }~\Bxi \times \Bn\right) \cdot \Bvarphi \, dS .
\end{array}
\end{equation}

On the other hand, 
\begin{equation}\label{4-0-16}  
\begin{array}{l}
\displaystyle \int_{\Omega} {\rm curl }~\Bxi \cdot {\rm curl}~\Bvarphi \, dx \\[3mm]
 \displaystyle = \int_{\Omega} {\rm curl}~\Bxi \cdot {\rm curl}~(\tilde{\Bvarphi} + \Bphi )\, dx \\[3mm]
\displaystyle = \int_{\Omega} \BJ \cdot \tilde{\Bvarphi}\, dx + \int_{\partial \Omega} ({\rm curl}~\Bxi \times \Bn ) \cdot \tilde{\Bvarphi} \, dS
 + \int_{\Omega} \BJ \cdot \Bphi \, dx + \int_{\partial \Omega} \BLambda \cdot \Bphi \, dS\\[3mm]
\displaystyle = \int_{\Omega} \BJ \cdot \tilde{\Bvarphi}\, dx + 
\int_{\Omega} J \cdot \Bphi \, dx + \int_{\partial \Omega} \BLambda \cdot \Bphi \, dS \\[3mm]
\displaystyle = \int_{\Omega} \BJ \cdot \Bvarphi \, dx + \int_{\partial \Omega} \BLambda \cdot \Bvarphi \, dS.
\end{array}
\end{equation}
Compare the two equalities \eqref{4-0-15} and \eqref{4-0-16}, we get that 
\begin{equation}\label{4-0-17}
{\rm curl}~\Bxi \times \Bn = \BLambda , \ \ \ \mbox{on}\ \partial \Omega.
\end{equation}
We conclude that $\Bxi$ is also a solution to the system \eqref{4-0-5}.
Moreover, let $\Bv = {\rm curl}~\Bxi$. 
\begin{equation}\nonumber
{\rm div}~\Bv = 0,\ \ \ \ {\rm curl}~\Bv = \BJ,\ \ \ \mbox{in}\ \Omega,\ \ \ \ \ \Bv\times \Bn = \BLambda,\ \ \ \mbox{on}\ \partial \Omega.
\end{equation}
It follows from the inequality \eqref{Sobolev2} and the estimate \eqref{4-0-11} that 
\begin{equation}\label{4-0-18}
\| \Bv\|_{W^{1, p}(\Omega)} \leq C \left( \| \Bv \|_{L^p(\Omega)} + \| \BJ \|_{L^p(\Omega)} + \|\BLambda\|_{W^{1-\frac1p, p}(\partial \Omega)}  \right) \leq C \left( \| \BJ \|_{L^p(\Omega)} + \|\BLambda\|_{W^{1-\frac1p, p}(\partial \Omega)} \right).
\end{equation}
Hence, applying the inequality \eqref{Sobolev1}, we have
\begin{equation}\label{4-0-19}
\|\Bxi \|_{W^{2, p}(\Omega) } \leq C \left( \|\Bxi\|_{L^p(\Omega)}  + \| \Bv\|_{W^{1, p}(\Omega)} \right) \leq 
C \left(\| \BJ \|_{L^p(\Omega)} + \|\BLambda\|_{W^{1-\frac1p, p }(\partial \Omega)}\right). 
\end{equation}

At last, let us discuss the uniqueness. If $\Bxi \in W^{2, p}(\Omega)$ is a solution to \eqref{4-0-5}, it is easy to check that 
$\Bxi$ is also a solution to the variational problem \eqref{4-0-9}. Hence the uniqueness of $\Bxi$ is indicated by the uniqueness of solutions to \eqref{4-0-9}.

\end{proof}

\begin{theorem}\label{theorem4-0-1}
Let $1 < p < +\infty$. 
Assume that $\BJ \in L^p(\Omega)$, $\BLambda \in W^{1- \frac{1}{p}, p}(\partial \Omega)$, and $(\BJ, \BLambda)$ satisfy the following 
compatibility conditions \eqref{compatibility-4-1}-\eqref{compatibility-4-2}. 
Then there exists a vector potential $\Bv \in W^{1, p}(\Omega)$ such that 
\begin{equation} 
\label{4-0-22}
\left\{
\begin{array}{l}
{\rm curl }~ \Bv = \BJ, \ \ \ \ \mbox{in}\ \Omega,\\[2mm]
{\rm div}~\Bv = 0,\ \ \ \ \mbox{in}\ \Omega,\\[2mm]
\Bv\times \Bn = \BLambda,\ \ \ \ \mbox{on}\ \partial \Omega.
\end{array}
\right.
\end{equation}
And it holds that 
\begin{equation}
\|\Bv\|_{W^{1, p}(\Omega)} \leq C \left(  \| \BJ \|_{L^p(\Omega)} + \|\BLambda \|_{W^{1 - \frac1p, p}(\partial \Omega) } \right). 
\end{equation}
\end{theorem}

\begin{proof}
Let $ \Bv= {\rm curl}~\Bxi$, where $\Bxi$ is the solution derived in Lemma \ref{lemma4-0-3}. $\Bv$ is the desired solution to \eqref{4-0-22}. 

\end{proof}

%%%%%%%%%%%%%%%%%%%%Solvability of div-curl%%%%%%%%%%%%%%%%%%%%%%%%%%%%%%%

In fact, we have a solvability result for the standard ${\rm div} - {\rm curl}$ system. 
\begin{theorem} \label{theorem4-0-2}
Let $1 < p < +\infty$. 
Assume that $\BJ \in L^p(\Omega)$, $\rho \in L^p(\Omega)$, $\BLambda \in W^{1- \frac{1}{p}, p }(\partial \Omega)$, and $(\BJ, \BLambda)$ satisfy the following 
compatibility conditions \eqref{compatibility-4-1}-\eqref{compatibility-4-2}. 
Then there exists one solution $\Bv \in W^{1, p}(\Omega)$ such that 
\begin{equation} \label{4-0-30}
\left\{
\begin{array}{l}
{\rm curl }~ \Bv = \BJ, \ \ \ \ \mbox{in}\ \Omega,\\[2mm]
{\rm div}~ \Bv = \rho,\ \ \ \ \mbox{in}\ \Omega,\\[2mm]
\Bv\times \Bn = \BLambda,\ \ \ \ \mbox{on}\ \partial \Omega,
\end{array}
\right.
\end{equation}
and 
\begin{equation}
\| \Bv \|_{W^{1, p}(\Omega)} \leq C \left(   \| \BJ \|_{L^p(\Omega)} + \|\rho \|_{L^p(\Omega)}+ \|\BLambda \|_{W^{1 - \frac1p, p}(\partial \Omega) }   \right). 
\end{equation}
Furthermore, if $ \BJ \in W^{m-1, p}(\Omega)$, $\rho \in W^{m-1, p}(\Omega)$, and $\BLambda \in W^{m - \frac{1}{p}, p}(\partial \Omega)$, $m \in \mathbb{N}^*$, then it holds that 
\begin{equation}\label{4-0-33}
\| \Bv \|_{W^{m, p}(\Omega) } \leq C \left(  \| \BJ \|_{W^{m -1, p}(\Omega)} + \|\rho \|_{W^{m -1, p}(\Omega) } 
+ \|\BLambda\|_{W^{m - \frac1p, p}(\partial \Omega) }   \right). 
\end{equation}

\end{theorem}

\begin{proof}
Let $q$ be the unique solution to the following Dirichlet problem 
\begin{equation}\label{4-0-31}
\left\{ \begin{array}{l}
\Delta q = \rho,\ \ \ \ \mbox{in}\ \Omega,\\[2mm]
q = 0,\ \ \ \ \mbox{on}\ \partial \Omega.
\end{array}
\right.
\end{equation}
It follows from the regularity theory for Laplace equation\cite{GT} that 
\begin{equation}\nonumber
\|\nabla q\|_{W^{1, p} (\Omega)} \leq C \|\rho \|_{L^p(\Omega)}.
\end{equation}

Suppose $\Bv_0$ is the solution to \eqref{4-0-22},  derived in Theorem \ref{theorem4-0-1}. Let $\Bv = \Bv_0 + \nabla q$. It is easy to check that 
$\Bv$ is a  solution to \eqref{4-0-30}, satisfying 
\begin{equation}\nonumber
\| \Bv\|_{W^{1, p}(\Omega)} \leq  C \left( \| \BJ \|_{L^p(\Omega)} + \|\rho \|_{L^p(\Omega)} + \|\BLambda\|_{W^{1-\frac1p, p}(\partial \Omega)}  \right). 
\end{equation}

The $W^{m p}$-estimate for $\Bv$ follows from the inequality \eqref{Sobolev2}.  
\end{proof}

\begin{remark} Alonso-Valli
\cite{AV} has considered the problem \eqref{4-0-30} in $W^{1, 2}$-space, which is a Hilbert space. Here we generalize their result to $W^{1, p}$-space. Note that $W^{1, p}(\Omega)$ is not a Hilbert space, the proof in \cite{AV} can not be applied to our case directly. Instead, our proof is inspired by \cite{AS, KY}, which revealed that the generalized Lax-Milgram theorem( 
Lemma \ref{lemma4-0-1}) is a powerful tool. 
\end{remark}

%%%%%%%%%%%%%%%%%%%%%%%%%%%%remark require amendent%%%%%%%%%%%%%%%%%%%%%%%%%

\begin{remark}\label{remark-compatibility-2}
The compatibility conditions \eqref{compatibility-4-1}-\eqref{compatibility-4-2} were proposed by Alonso-Valli\cite{AV}. In fact, \eqref{compatibility-4-2} can be replaced by another compatibility condition
\begin{equation} \label{compatibility-4-7}
< \BJ \cdot \Bn_j, \, 1>_{\Sigma_j}= \int_{\partial \Sigma_j} (\Bn \times \BLambda ) \cdot \Btau_j \,  dl,\ \ \ \ \  1 \leq j \leq N_2.
\end{equation}

\end{remark}

\begin{proof}
On one hand, assume  the conditions \eqref{compatibility-4-1} and \eqref{compatibility-4-7} hold, then
\begin{equation}\nonumber
\begin{array}{l}
\displaystyle \int_{\Omega} \BJ \cdot \nabla^0 q_{j, *}^T\, dx \\[3mm]
= \displaystyle \int_{\Omega^0} \BJ \cdot \nabla^0  q_{j, *}^T\, dx \\[3mm]
=- \displaystyle \int_{\Omega^0} ( \nabla^0\cdot \BJ) \cdot q_{j, *}^T\, dx 
+ \int_{\partial \Omega} ( \BJ \cdot \Bn) \cdot q_{j, *}^T\, dS + \sum_{k =1}^{N_2}
 \int_{\Sigma_k}( \BJ \cdot \Bn_k) [q_{j, *}^T]_k \, dS \\[3mm]
= \displaystyle  \int_{\partial \Omega} (\BJ \cdot \Bn) \cdot q_{j, *}^T\, dS   
+   \sum_{k =1}^{N_2}   < \BJ \cdot \Bn_k,\ 1>_{\Sigma_k} [ q_{j, *}^T]_k    \\[3mm]
= \displaystyle  \int_{\partial \Omega} {\rm div}_T~\BLambda \cdot q_{j, *}^T\, dS 
+   \sum_{k =1}^{N_2}   < \BJ \cdot \Bn_k,\ 1>_{\Sigma_k} [ q_{j, *}^T]_k   \\[3mm]
= \displaystyle - \int_{\partial \Omega} \BLambda \cdot \nabla^0_T q_{j, *}^T \, dS
+ \sum_{k =1}^{N_2} \int_{\partial \Sigma_k} \BLambda \cdot (  \Bn \times \Btau_k) \, dl \cdot [ q_{j, *}^T]_k
+   \sum_{k =1}^{N_2}   < \BJ \cdot \Bn_k,\ 1>_{\Sigma_k} [ q_{j, *}^T]_k
\\[3mm]
=  \displaystyle - \int_{\partial \Omega} \BLambda \cdot \nabla^0 q_{j, *}^T \, dS  
 - \sum_{k =1}^{N_2} \left[ \int_{\partial \Sigma_k}  ( \Bn \times \BLambda ) \cdot \Btau_k \, dl  - < \BJ \cdot \Bn_k,\ 1>_{\Sigma_k} \right] \cdot  [ q_{j, *}^T]_k\\[3mm]
=  \displaystyle - \int_{\partial \Omega} \BLambda \cdot \nabla^0 q_{j, *}^T \, dS  . 
\end{array}
\end{equation}
Since $\{ \nabla^0 q_{j, *}^T ; \ 1\leq j \leq N_2\}$ is a basis of $K_T(\Omega)$, it holds that 
\begin{equation}\nonumber
\int_{\Omega} \BJ \cdot \Bvarphi \, dx = - \int_{\partial \Omega} \BLambda \cdot \Bvarphi \, dS,\ \ \ \ \forall \ \Bvarphi \in K_T(\Omega).
\end{equation}

On the other hand, suppose the compatibility conditions \eqref{compatibility-4-1}- \eqref{compatibility-4-2} hold. For every 
fixed $1\leq j \leq N_2$, choose some function $r_{j, *}^T \in \Theta $ with 
\begin{equation}\nonumber
\nabla^0 r_{j, *}^T \in K_T(\Omega),\ \ \ \ \ [r_{j, *}^T]_k = \delta_{jk}. 
\end{equation} 
The functions $\{ r_{j, *}^T ; \, 1 \leq j \leq N_2 \}$ have been constructed in \cite{KY}( In fact, they also span one basis of $K_T(\Omega)$). 
\begin{equation} \begin{array}{l}
\displaystyle \int_{\Omega} \BJ \cdot \nabla^0 r_{j, *}^T \, dx \\[3mm]= 
\displaystyle
\int_{\partial \Omega} ( \BJ \cdot \Bn)  \cdot r_{j, *}^T \, dS + \sum_{k=1}^{N_2} \int_{\Sigma_k} ( \BJ \cdot \Bn_k) \cdot 
[r_{j, *}^T]_k  \, dS \\[3mm] 
\displaystyle = \int_{\partial \Omega} {\rm div}_T \BLambda \cdot r_{j, *}^T \, dS 
+ <\BJ \cdot \Bn_j, \, 1>_{\Sigma_j} \\[3mm] 
\displaystyle  = - \int_{\partial \Omega} \BLambda \cdot \nabla^0  r_{j, *}^T \, dS 
+ \int_{\partial \Sigma_j} \BLambda \cdot ( \Bn \times \Btau_j) \, dl      +
<\BJ \cdot \Bn_j, \, 1>_{\Sigma_j},
\end{array}
\end{equation}
which implies that 
\begin{equation}\nonumber
<\BJ \cdot \Bn_j, \, 1>_{\Sigma_j} = \int_{\partial \Sigma_j} ( \Bn \times \BLambda ) \cdot \Btau_j \, dl ,\ \ \ \ 1\leq j \leq N_2. 
\end{equation}

\end{proof}

%%%%%%%%%%%%%%%Proof of Theorem 2.3%%%%%%%%%%%%%%%%%%%%%%%%%%%%%%%%%%%%%%%%%%%%%%%%%%%%%%%%%%%%%5

\subsection{Proof of Theorem \ref{theorem3}}

\begin{proof}[\bf Proof of Theorem \ref{theorem3}] Suppose $\Bv$ is a solution derived in Theorem \ref{theorem4-0-1} to the system \eqref{4-0-22}.  Let us consider the following system
\begin{equation}\label{4-3-20}
\left\{  \begin{array}{l}
{\rm curl }~\Bw = 0, \ \ \ \mbox{in}\ \Omega, \\[2mm]
{\rm div}~(\Bepsilon \Bw) = \rho  - {\rm div}~(\Bepsilon \Bv),\ \ \ \mbox{in}\ \Omega,\\[2mm]
\Bw\times \Bn = 0,\ \  \ \mbox{on}\ \partial \Omega.
\end{array}
\right.
\end{equation}
Since ${\rm curl }~\Bw = 0$ in $\Omega$, and $\Bw\times \Bn = 0$ on $\partial \Omega$, $\Bw$ is in fact a gradient, i.e., $\Bw = \nabla q$ with $q= constant$ on every $\Gamma_i$, $0\leq i \leq N_1$.  Let us find one special solution.  Consider the following elliptic system
\begin{equation}\label{4-3-21}
\left\{ \begin{array}{l}
{\rm div}~(\Bepsilon \nabla q ) = \rho  - {\rm div}~(\Bepsilon \Bv), \ \ \ \mbox{in}\ \Omega, \\[2mm]
q = 0,\ \ \ \mbox{on}\ \partial \Omega.
\end{array}
\right.
\end{equation}
\eqref{4-3-21} is  a classical elliptic equation with Dirichlet boundary condition.
There exists one unique solution $q \in W^{m+1, p}(\Omega)$, with the estimate
\begin{equation}\label{4-3-22}
\begin{array}{ll}
\|\nabla q\|_{W^{m,p}(\Omega)} & \leq C \|\rho - {\rm div}~(\Bepsilon \Bv)\|_{W^{m -1, p}(\Omega)}\\[2mm]
& \leq C \|\rho\|_{W^{m-1, p}(\Omega)} + C \| \Bv\|_{W^{m, p}(\Omega)} \\[2mm]
& \leq C\left(  \| \BJ \|_{W^{m-1, p}(\Omega)} + \|\rho \|_{W^{m-1, p}(\Omega)} + \|\BLambda \|_{W^{m - \frac1p, p}(\partial \Omega)}                     \right),
\end{array}
\end{equation}
where the last inequality is due to Theorem \ref{theorem4-0-1}.

Let $\Bu_0 = \nabla q +  \Bv$, it is easy to check that $\Bu_0$ is a solution to \eqref{1-2}, and
\begin{equation}\label{4-3-23}
\|\Bu_0 \|_{W^{m, p}(\Omega)} \leq C \left(  \| \BJ \|_{W^{m-1, p}(\Omega)} + \|\rho \|_{W^{m-1, p}(\Omega)} + \|\BLambda \|_{W^{m - \frac1p, p}(\partial \Omega)}                     \right).
\end{equation}

That ends the proof of Theorem \ref{theorem3}.
\end{proof}

%%%%%%%%%%%%%%%%%%%%%%%%%%%%%Null space%%%%%%%%%%%%%%%%

\subsection{The null space $K_{N, \Bepsilon}^p(\Omega)$}

In this subsection, we will talk about the null space $K_{N, \Bepsilon}^p(\Omega)$. One particular basis will be given. The characterization of $K_{N, \Bepsilon}^p(\Omega)$ is also inspired by that of $K_N^p(\Omega)$ in \cite{ABDG, AS, KY}.

\begin{theorem}\label{theorem4.3} The dimension of the null space $K_{N, \Bepsilon}^2(\Omega)$ is equal to the Betti number $N_1$. It is spanned by the functions $\left\{\nabla q_i^N; \ 1\leq i \leq N_1 \right\}$, where each $q_i^N$ is the unique solution in $W^{1, 2}(\Omega)$ of the problem
\begin{equation}\label{4-2-1}
\left\{ \begin{array}{l}
- {\rm div}~(\Bepsilon \nabla q_i^N) = 0, \ \ \ \mbox{in}\ \Omega,\\[2mm]
q_i^N |_{\Gamma_0} = 0,\ \ \ \mbox{and}\ \ \ q_i^N|_{\Gamma_k} = constant,\  \ 1 \leq k \leq N_1, \\[2mm]
<\Bepsilon \nabla q_i^N \cdot \Bn,\ 1>_{\Gamma_k} = \delta_{ik}, \ \ 1\leq k \leq N_1, \ \ \ \ <\Bepsilon \nabla q_i^N \cdot \Bn, \ 1>_{\Gamma_0} = -1.
\end{array}
\right.
\end{equation}
\end{theorem}

\begin{proof}We decompose the proof into two parts. The first part is devoted to the existence of $q_i^N$, while in the second part we verify that $\{ \nabla q_i^N; \ 1\leq i \leq N_1  \}$ is a 
basis of $K_{N, \Bepsilon}^2(\Omega)$. 

{\bf Part I}\  \ First, let us define one function space,
\begin{equation}\nonumber
\Theta^0 = \left\{  r \in W^{1, 2}(\Omega):\ r|_{\Gamma_0} = 0,\ \ r|_{\Gamma_i} = constant, \ 1\leq i \leq N_1             \right\}.
\end{equation}
For every $1\leq i \leq N_1$, let us consider the following variational problem: \ find $q_i^N$ in $\Theta^0$ such that for every $r \in \Theta^0$,
\begin{equation}\label{4-2-5}
\int_{\Omega} \Bepsilon \nabla q_i^N \cdot \nabla r \, dx = r|_{\Gamma_i}.
\end{equation}
Applying Lax-Milgram theorem, we deduce that \eqref{4-2-5} has a unique solution $q_i^N \in \Theta^0$.

For every function $r \in C_0^\infty(\Omega)$,
\begin{equation}\nonumber
\int_{\Omega} \Bepsilon \nabla q_i^N \cdot \nabla r \, dx  = 0,
\end{equation}
which implies that
\begin{equation}\label{4-2-7}
- {\rm div}~(\Bepsilon \nabla q_i^N ) = 0,\ \ \ \mbox{in}\ \Omega.
\end{equation}
For every fixed $i$, $1\leq i \leq N_1$, choose some function $r_i \in W^{1, 2}(\Omega)$, satisfying
\begin{equation}\nonumber
r_i = 1 \ \ \ \mbox{ on} \ \Gamma_i \ \ \ \ \mbox{and} \ \ \ \ \ \ r_i  = 0\ \ \ \mbox{ on}\ \Gamma_j,\ \ j \neq i.
\end{equation}
 Then
\begin{equation}\label{4-2-8}
1 = r_i |_{\Gamma_i} = \int_{\Omega} \Bepsilon \nabla q_i^N \cdot \nabla r_i \, dx = < \Bepsilon \nabla q_i^N\cdot \Bn, \ 1>_{\Gamma_i}.
\end{equation}
Moreover, for every $1\leq j \leq N_1$, $j \neq i$,
\begin{equation}\label{4-2-9}
0 = r_j |_{\Gamma_i} = \int_{\Omega} \Bepsilon \nabla q_i^N \cdot \nabla r_j \, dx = < \Bepsilon \nabla q_i^N \cdot \Bn , \ 1>_{\Gamma_j}.
\end{equation}
Similarly, choose some function $\tilde{r_i} \in W^{1, 2}(\Omega)$ such that
\begin{equation}\nonumber
\tilde{r_i} = 1, \ \ \ \mbox{on}\ \ \Gamma_i \ \ \mbox{and}\ \Gamma_0,\ \ \ \ \tilde{r_i} = 0\ \ \mbox{on}\ \Gamma_j, \ j\neq i, \ 1\leq j \leq N_1.
\end{equation}
We can easily deduce that
\begin{equation}\label{4-2-10}
<\Bepsilon \nabla q_i^N \cdot \Bn,\ 1>_{\Gamma_0} = -1.
\end{equation}
The above argument verifies that $q_i^N$ is in fact a solution to \eqref{4-2-1}. On the other hand, every solution of \eqref{4-2-1} solves \eqref{4-2-5}. Thus \eqref{4-2-1} admits one unique solution $q_i^N \in \Theta^0$.

{\bf Part II}\ \ The functions $\left\{\nabla q_i^N:  \ 1\leq i \leq N_1 \right\}$ are obviously independent. It remains to prove that they span $K_{N, \Bepsilon}^2(\Omega)$. Take any
function $\Bu \in K_{N, \Bepsilon }^2(\Omega)$ and consider the function
\begin{equation}\nonumber
\Bw = \Bu - \sum_{i =1}^{N_1} \left<  \Bepsilon u \cdot \Bn, \ 1 \right>_{\Gamma_i} \nabla q_i^N.
\end{equation}
It is easy to check that
\begin{equation}\label{4-2-12}
{\rm div}~(\Bepsilon \Bw) = 0,\ \ \ \mbox{in}\ \Omega,\ \ \ \mbox{and}\ \ \ \left< \Bepsilon \Bw \cdot \Bn, \ 1 \right>_{\Gamma_i} = 0,\ \ 0\leq i \leq N_1.
\end{equation}
According to Lemma \ref{lemma3-1}, there exists a vector potential $\Bpsi \in W^{1, 2}(\Omega)$, such that
\begin{equation}\nonumber
\Bepsilon \Bw = {\rm curl }~\Bpsi,\ \ \ \mbox{in}\ \Omega.
\end{equation}
And hence,
\begin{equation}\nonumber
\int_{\Omega} \Bepsilon \Bw \cdot \Bw \, dx = \int_{\Omega} {\rm curl }~\Bpsi \cdot \Bw \, dx =  \int_{\Omega} \Bpsi \cdot {\rm curl}~\Bw\, dx + \int_{\partial \Omega}
(\Bw\times \Bn) \cdot \Bpsi \, dS = 0,
\end{equation}
which implies $\Bw\equiv 0$ in $\Omega$. That ends the proof of Theorem \ref{theorem4.3}.

\end{proof}

%%%%%%%%%%%%%%%%%%%%%%%%%%%%%%%%%%%%%%%%Indentity between K_N^p and K_N^2%%%%%%%%%%%%%%%%%%%%%%%%%%%%%%%%%55
Next, we will prove the identity between $K_{N, \Bepsilon}^p(\Omega) $ and $K_{N, \Bepsilon}(\Omega)$. 
\begin{theorem}\label{Nullspace2}
For every $1<p< \infty$, $K_{N, \Bepsilon}^p(\Omega) = K_{N, \Bepsilon}(\Omega)$. 
\end{theorem}

\begin{proof}
For every $1< p < \infty$, $K_{N, \Bepsilon} (\Omega) \subseteq K_{N,  \Bepsilon}^p(\Omega)$. Hence,  it remains to prove that for every function $\Bu \in K_{N, \Bepsilon}^p(\Omega)$, 
 $u \in C^\infty(\overline{\Omega})$. 
Since 
\begin{equation}\nonumber
{\rm curl}~\Bu=0,\ \ \ \mbox{in}\ \Omega,\ \ \ \ \mbox{and}\ \ \ \ \Bu\times \Bn = 0\ \ \ \mbox{on}\ \ \partial \Omega,
\end{equation}
there exists one function $q \in W^{1, p}(\Omega)$, such that
\begin{equation}\nonumber
\Bu = \nabla q,\ \ \  \ \ \mbox{in}\ \ \Omega.
\end{equation}
Moreover, $q =constant$ on $\Gamma_i$, \ $0\leq i \leq N_1$.

On the other hand, 
\begin{equation}\nonumber
{\rm div}~(\Bepsilon \nabla q ) = {\rm div}~(\Bepsilon \Bu ) = 0,\ \ \ \ \mbox{in}\ \ \Omega.
\end{equation}
It follows from the classical regularity theory for elliptic equation with Dirichlet boundary condition\cite{GT} that 
$q \in C^\infty(\overline{\Omega})$. That ends the proof of Theorem \ref{Nullspace2}.

\end{proof}

\begin{remark}
When the domain $\Omega$ is of class $C^{2}$, following the same line as above, we can also prove that $K_{N, \Bepsilon}^p(\Omega) = K_{N, \Bepsilon}^2(\Omega).$
\end{remark}

%%%%%%%%%%%%%%%%%%%%%%%%%%%%%%%Proof of Theorem 4%%%%%%%%%%%%%%%%%%%%%%%%%%%%%%%

\subsection{Proof of Theorem \ref{theorem4}}  

In this subsection, we will prove Friedrichs inequality involving tangential boundary value. As before, our proof is based on the solvability of \eqref{1-2} and the characterization of the null space $K_{N, \Bepsilon}^p(\Omega)$. 
Let
\begin{equation}\nonumber
\BJ = {\rm curl }~\Bu, \ \ \ \rho = {\rm div}~(\Bepsilon \Bu), \ \ \ \BLambda = \Bu \times \Bn.
\end{equation}
According to Theorem \ref{theorem3}, there exists a function $\Bu_0 \in W^{m, p}(\Omega)$, such that
\begin{equation} \nonumber
\left\{ \begin{array}{l}
{\rm curl}~\Bu_0 = J, \ \ \ \mbox{in}\ \Omega, \\[2mm]
{\rm div}~(\Bepsilon \Bu_0) = \rho, \ \ \ \mbox{in}\ \Omega, \\[2mm]
\Bu_0 \times \Bn = \BLambda, \ \ \ \mbox{on}\ \partial \Omega,
\end{array}
\right.
\end{equation}
with the estimate
\begin{equation}\label{4-3-0}
\|\Bu_0 \|_{W^{m, p}(\Omega)} \leq C \left( \|\BJ\|_{W^{m-1, p}(\Omega)} + \|\rho\|_{W^{m-1, p}(\Omega)} + \|\BLambda\|_{W^{m - \frac1p, p}(\partial \Omega)}        \right).
\end{equation}

Let $\Bv = \Bu  - \Bu_0 $. Then $\Bv \in K_{N, \Bepsilon}^p(\Omega) = K_{N, \Bepsilon}(\Omega)$. As revealed in the proof of Theorem \ref{theorem4.3},
\begin{equation}\label{4-3-1}
\Bv = \sum_{i =1}^{N_1}  <  \Bepsilon \Bv \cdot \Bn, \ 1>_{\Gamma_i} \nabla q_i^N,
\end{equation}
where $\left\{ \nabla q_i^N; \ 1\leq i \leq N_1 \right \}$ is the basis for $K_{N, \Bepsilon}^2(\Omega)$, derived in Subsection 4.3. Due to Theorem \ref{Nullspace2}, $\nabla q_i^N \in C^\infty(\overline{\Omega})$.  Denote $\tilde{C_i}= \|\nabla q_i^N \|_{W^{m, p}(\Omega)}< + \infty$,

\begin{equation}\label{4-3-2} \begin{array}{ll}
\|\Bv\|_{W^{m, p}(\Omega)} &  \displaystyle \leq \sum_{i =1}^{N_1} \left\|\Bepsilon \Bv \cdot \Bn  \right\|_{W^{-\frac1p, p}(\Gamma_i)} \cdot \tilde{C_i}\\[3mm]
& \leq C\left \|\Bepsilon \Bv \right\|_{L^p(\Omega)} + C \left\| {\rm div}~(\Bepsilon \Bv) \right\|_{L^p(\Omega)} \\[3mm]
&=  C \left\|\Bepsilon \Bv \right\|_{L^p(\Omega)} \\[3mm]
&\leq  C \left( \|\Bu\|_{L^p(\Omega)}   + \|\Bu_0 \|_{L^p(\Omega) } \right).
\end{array}
\end{equation}
Consequently, $\Bu \in W^{m, p}(\Omega)$ and combining the estimates \eqref{4-3-1}-\eqref{4-3-2}, we have
\begin{equation}\nonumber
\|\Bu\|_{W^{m, p}(\Omega)} \leq C \left(\|\Bu\|_{L^p(\Omega)} + \|{\rm curl}~\Bu \|_{W^{m-1, p}(\Omega)} + \|{\rm div}~(\Bepsilon \Bu)\|_{W^{m -1, p}(\Omega)}
+ \|\Bu \times \Bn \|_{W^{m - \frac1p, p}(\partial \Omega)}                  \right).
\end{equation}
That ends the proof of Theorem \ref{theorem4}.

\begin{remark}
According to the proof of Theorem \ref{theorem4}, we can get another type of estimate for $\Bu$,
\begin{equation} \begin{array}{ll}
\|\Bu\|_{W^{m, p}(\Omega)} & \leq \|\Bu_0 \|_{W^{m, p}(\Omega)} + \|\Bv\|_{W^{m, p}(\Omega)} \\[2mm]
& \leq C \left(   \|{\rm curl}~\Bu \|_{W^{m-1, p}(\Omega)} + \|{\rm div}~(\Bepsilon \Bu)\|_{W^{m -1, p}(\Omega)}
+ \|\Bu \times \Bn \|_{W^{m - \frac1p, p}(\partial \Omega)}      \right)  \\[3mm]
&\ \ \ \ \ \  \ \ \ \ \ \ \ \displaystyle + C\sum_{i =1}^{N_1} \left| < \Bepsilon \Bv \cdot \Bn, 1>_{\Gamma_i } \right|\\[3mm]
& \leq C \left(   \|{\rm curl}~\Bu \|_{W^{m-1, p}(\Omega)} + \|{\rm div}~(\Bepsilon \Bu)\|_{W^{m -1, p}(\Omega)}
+ \|\Bu \times \Bn \|_{W^{m - \frac1p, p}(\partial \Omega)}      \right)  \\[3mm]
&\ \ \ \ \ \ \ \ \ \ \ \ \displaystyle + C \sum_{i =1}^{N_1} \left| < \Bepsilon \Bu \cdot \Bn, 1>_{\Gamma_i } \right|
\end{array}
\end{equation}
\end{remark}

\begin{remark}
Theorem \ref{theorem4} holds for the domain which is of class $C^{m+1}$. 
\end{remark}

%%%%%%%%%%%%%%%%%%%%%%%%%%%%%%%%%%%%%%Section 5%%%%%%%%%%%%%%%%%%%%%%%%%%%%%%%%%%%%%%%%%%%%%%%%%%%%%%%%

\section{Generalized Helmholtz-Weyl Decompositions}

In this section, we give two decompositions of vector fields $\Bu \in L^2(\Omega)$. The decompositions are designed for solvability of Maxwell equations. They have been 
discussed in many papers\cite{Picard, Saranen82, Saranen83}. We will give a new description. Moreover, we will talk about the decompositions in $W^{m, p}$-spaces. 

Define the function spaces $W^{1, 2}_{{\rm div}^0}(\Omega)$ and $W^{1, 2}_N(\Omega)$,
\begin{equation}\nonumber
W^{1, 2}_{{\rm div}^0} (\Omega) = \left\{  \Bv \in W^{1, 2}(\Omega):\ {\rm div}~\Bv = 0\ \ \ \mbox{in}\ \Omega \right\}.
\end{equation}
\begin{equation}\nonumber
W^{1, 2}_N(\Omega) = \left\{ \Bv \in W^{1, 2}(\Omega):\ \Bv\times \Bn = 0 \ \ \ \mbox{on}\ \partial \Omega, \ \ \mbox{and}\ \ <\Bv \cdot \Bn, 1>_{\Gamma_i} = 0,\ \ 0\leq i \leq N_1 \right\}.
\end{equation}

\begin{theorem}\label{decomposition1}

Let $\Bu \in L^2(\Omega)$. Then there exists $\Bh \in K_{T, \Bsigma}^2(\Omega)$, $\chi \in W^{1, 2}(\Omega)$, $\Bw \in W_{{\rm div}^0}^{1, 2}(\Omega) \cap W^{1, 2}_N(\Omega)$,  such that $u$ can be represented as 
\begin{equation}\nonumber
u = \Bh  + \Bsigma^{-1} \nabla \chi + {\rm curl}~\Bw , 
\end{equation}
where $\Bh$ is unique, $\chi$ is unique up to an additive constant and $\Bw$ is unique. Moreover, we have the
estimate:
\begin{equation}\nonumber
\|\Bh\|_{L^2(\Omega)} + \| \nabla \chi\|_{L^2(\Omega)} + \|\Bw\|_{W^{1,2}(\Omega)} \leq C \| \Bu\|_{L^2(\Omega)}.
\end{equation} 

\end{theorem}

Before the proof of Theorem \ref{decomposition1}, let us introduce one lemma, which is a particular case of Theorem \ref{theorem4-0-1}. The proof can also be found in \cite{AS}.

\begin{lemma}\label{asky} Let $1< p < + \infty$. 
A function $\Bz \in L^p({\rm div}; \Omega)$ satisfies
\begin{equation}\label{compatible-condition}
{\rm div}~\Bz = 0\ \ \mbox{in}\ \Omega, \ \ \ \Bz\cdot \Bn = 0\ \ \mbox{on} \ \partial \Omega,\ \ \ \mbox{and}\ \ <\Bz\cdot \Bn_j, 1>_{\Sigma_j} = 0, \ 1\leq j\leq N_2,
\end{equation}
if and only if there exists a vector potential $\Bpsi \in W^{1, p}(\Omega)$ such that
\begin{equation}\nonumber\begin{array}{l}
\Bz = {\rm curl}~\Bpsi , \ \ \ \ \mbox{and}\ \ \ {\rm div}~\Bpsi =0,\ \ \mbox{in}\ \Omega, \\[2mm]
\Bpsi \times \Bn = 0\ \ \ \ \mbox{on}\ \partial \Omega,\\[2mm]
<\Bpsi \cdot \Bn , 1>_{\Gamma_i}= 0,\ \ \ \mbox{for every} \ \ 0\leq i \leq N_1.
\end{array}
\end{equation}
This function $\Bpsi$ is unique and we have the estimate
\begin{equation}\nonumber
\|\Bpsi \|_{W^{1, p}(\Omega)} \leq C \|\Bz\|_{L^p(\Omega)}.
\end{equation}
\end{lemma}

\begin{proof}[\bf Proof of Theorem \ref{decomposition1}] First, 
 the scalar potential $\chi \in W^{1, 2}(\Omega)$ is taken as a variational solution of the following problem:
\begin{equation}\nonumber
\forall\  \mu \in W^{1,2}(\Omega), \ \ \ \ \  \int_{\Omega} \Bsigma^{-1} \nabla \chi \cdot \nabla \mu \, dx = \int_{\Omega} \Bu \cdot \nabla \mu \, dx. 
\end{equation}
Such a scalar function $\chi$ is unique up to an additive constant, due to Lax-Milgram theorem. And it holds that 
\begin{equation}\label{5-1}
\|\nabla \chi \|_{L^2(\Omega)} \leq C \| \Bu \|_{L^2(\Omega)}.
\end{equation}
Moreover, the variational equality implies that 
\begin{equation}\label{5-2}
{\rm div}~( \Bu - \Bsigma^{-1} \nabla \chi ) = 0,\ \ \ \mbox{in}\ \Omega,\ \ \ \ \mbox{and}\ \ \ \ (\Bu - \Bsigma^{-1} \nabla \chi )\cdot \Bn = 0,\ \ \ \mbox{on}\ \partial \Omega. 
\end{equation}

Next, the $\Bsigma$-harmonic field $\Bh$ is chosen in the following way, 
\begin{equation}\nonumber \left\{
\begin{array}{l}
{\rm curl}~(\Bsigma \Bh) = 0,\ \ \ \mbox{in}\ \Omega,\\[2mm]
{\rm div}~\Bh = 0,\ \ \ \mbox{in}\ \Omega,\\[2mm]
\Bh \cdot \Bn = 0,\ \ \ \mbox{on}\ \partial \Omega,\\[2mm]
<\Bh\cdot \Bn_j, 1>_{\Sigma_j} = < (\Bu - \Bsigma^{-1} \nabla \chi)\cdot \Bn_j, 1>_{\Sigma_j },\ \ \ 1\leq j \leq N_2.
\end{array} \right.
\end{equation}
In fact, according to the proof of Theorem \ref{lemma3-3}, 
\begin{equation} \label{5-5}
\Bh = \sum_{j=1}^{N_2} < ( \Bu - \Bsigma^{-1} \nabla \chi)\cdot \Bn_j, 1>_{\Sigma_j}  \Bsigma^{-1} \nabla^0 q_j^T,
\end{equation}
where $\{ \Bsigma^{-1} \nabla^0 q_j^T;\ 1\leq j \leq N_2 \}$ is a basis for $K_{T, \Bsigma}^2(\Omega)$, derived in Subsection 3.2. Hence, 
\begin{equation}\nonumber
\| \Bh\|_{L^2(\Omega)}  \leq C \|\Bu - \Bsigma^{-1} \nabla \chi \|_{L^2(\Omega)} + C \|{\rm div}~( \Bu - \Bsigma^{-1}\nabla \chi) \|_{L^2(\Omega)}
\leq C \| \Bu\|_{L^2(\Omega)}, 
\end{equation}
due to the fact that ${\rm div}~ ( \Bu - \Bsigma^{-1} \nabla \chi ) = 0$ in $\Omega$.

Then let us define $\Bz =  \Bu - \Bh  - \Bsigma^{-1}\nabla \chi$. It is easy to check that $\Bz\in L^2(\Omega)$ satisfies that 
\begin{equation}\nonumber \left\{ 
\begin{array}{l}
{\rm div}~ \Bz = 0,\ \ \ \mbox{in}\ \Omega,\\[2mm]
\Bz \cdot \Bn = 0,\ \ \ \mbox{on}\ \partial \Omega,\\[2mm]
< \Bz \cdot \Bn_j, 1>_{\Sigma_j} = 0, \ \ \ 1 \leq j \leq N_2.
\end{array}
\right.
\end{equation}

It follows from Lemma \ref{asky} that there exists a unique vector potential $\Bw \in W^{1,2}_{{\rm div}^0}(\Omega) \cap W^{1, 2}_N(\Omega)$ such that 
\begin{equation}\nonumber
\Bz = {\rm curl}~\Bw ,\ \ \ \mbox{in}\ \Omega. 
\end{equation}
And it holds that 
\begin{equation}\nonumber
\|\Bw\|_{W^{1, 2}(\Omega)} \leq C \|\Bz\|_{L^2(\Omega)} \leq C \|\Bu\|_{L^2(\Omega)}.
\end{equation}

At last, let us prove the uniqueness of the above decomposition. Suppose $\Bu$ has another decomposition, 
\begin{equation}\nonumber
\Bu = \tilde{\Bh} + \Bsigma^{-1} \nabla \tilde{\chi} + {\rm curl}~\tilde{\Bw},
\end{equation}
where $\tilde{\Bh}\in K_{T, \Bsigma}^2(\Omega)$, $\tilde{\chi} \in W^{1, 2}(\Omega)$, and $\tilde{\Bw} \in W_{ {\rm div}^0 }^{1, 2}(\Omega) 
\cap W_{N}^{1, 2}(\Omega)$.  Then 
\begin{equation}\label{5-9}
0 = (\Bh - \tilde{\Bh}) + \Bsigma^{-1} \nabla ( \chi - \tilde{\chi} ) + {\rm curl}~( \Bw - \tilde{\Bw} ).
\end{equation}
Taking the $L^2$-inner product of \eqref{5-9} and $\Bsigma ( \Bh - \tilde{\Bh} )$, we have
\begin{equation}\nonumber
0 = \left< \Bh - \tilde{\Bh},\ \Bsigma ( \Bh - \tilde{\Bh} )  \right>, 
\end{equation}
which implies 
\begin{equation}\nonumber
\Bh = \tilde{\Bh}\ \ \ \mbox{in}\ \Omega.
\end{equation}
Similarly, taking the $L^2$-inner product of \eqref{5-9}   and $\nabla (\chi - \tilde{\chi} )$, $\Bsigma {\rm curl}~(\Bw - \tilde{\Bw} )$ resp. , we have
\begin{equation}\nonumber
\nabla (\chi - \tilde{\chi}) = 0, \ \ \ \mbox{and}\ \  \ {\rm curl}~(\Bw - \tilde{\Bw} ) = 0,\ \ \ \mbox{in}\ \Omega.
\end{equation}
It completes the proof of Theorem \ref{decomposition1}.

\end{proof}

\begin{remark}
Theorem \ref{decomposition1} can be conculded in a simple formula, 
\begin{equation}\nonumber
L^2(\Omega) = K_{T, \Bsigma}^2(\Omega) \oplus \Bsigma^{-1} \nabla W^{1, 2}(\Omega) \oplus {\rm curl}~\left(W^{1, 2}_{{\rm div}^0 }(\Omega) \cap W_{N}^{1, 2}(\Omega) \right). 
\end{equation}
\end{remark}

The above decomposition can be generalized to $W^{m, p}$-spaces. 

\begin{theorem}\label{decomposition2}
Let $\Bu \in W^{m, p}(\Omega)$, $m \in \mathbb{N}^*$, $1 < p < + \infty$. The decomposition for $\Bu$ in Theorem \ref{decomposition1} satisfies the following estimate:
\begin{equation} \label{5-11}
\|\Bh\|_{W^{m, p}(\Omega)} + \|\nabla \chi\|_{W^{m, p}(\Omega)} + \|\Bw\|_{W^{m+1, p}(\Omega)} \leq C \|\Bu \|_{W^{m, p}(\Omega)}.
\end{equation}
\end{theorem}

\begin{proof}
When $\Bu$ belongs to $W^{m, p}(\Omega)$, $\chi$ is in fact a solution to the following elliptic equation, 
\begin{equation}\nonumber
\left\{
\begin{array}{l}
{\rm div}~(\Bsigma^{-1} \nabla \chi) = {\rm div}~\Bu ,\ \ \ \mbox{in}\ \Omega,\\[2mm]
\Bsigma^{-1} \nabla \chi \cdot \Bn = \Bu \cdot \Bn, \ \ \ \mbox{on}\ \partial \Omega.
\end{array}
\right.
\end{equation}
It follows from the classical regularity theory for elliptic equations with conormal boundary condition\cite{Lieberman}, that 
\begin{equation}\nonumber
\|\nabla \chi\|_{W^{m, p}(\Omega)} \leq C \| \Bu \|_{W^{m, p}(\Omega)}.
\end{equation}

On the other hand, 
\begin{equation}\nonumber
\| \Bh \|_{W^{m, p}(\Omega)} \leq C \sum_{j =1}^{N_2} \left| < ( \Bu - \Bsigma^{-1} \nabla \chi) \cdot \Bn_j, \, 1>_{\Sigma_j} \right| 
\leq C \left\| \Bu  - \Bsigma^{-1}\nabla \chi \right\|_{W^{m, p}(\Omega)} \leq C \| \Bu \|_{W^{m, p}(\Omega)}.
\end{equation} 
And the estimate for $\Bw$ follows from Lemma \ref{asky} and the inequality \eqref{Sobolev2}.

\end{proof}

\vspace{3mm}
Define the function space $W^{1, 2}_T(\Omega)$, 
\begin{equation} \nonumber
W_T^{1, 2}(\Omega)   = \left\{ \Bv \in W^{1, 2}(\Omega):\ \Bv \cdot \Bn = 0, \ \ \mbox{on}\ \partial \Omega, \ \ \ \ < \Bv \cdot \Bn_j, 1>_{\Sigma_j}= 0, \ 1\leq j \leq N_2                \right\}. 
\end{equation}

\begin{theorem}\label{decomposition3}
Let $\Bu\in L^2(\Omega)$. Then there exist $\Bh \in K_{N, \Bepsilon}^2(\Omega)$, $\chi \in W_0^{1, 2}(\Omega)$, and $\Bw\in W_{{\rm div}^0}^{1, 2}(\Omega) \cap W_T^{1, 2}(\Omega)$, such that $\Bu$ can be represented as
\begin{equation}\nonumber
\Bu = \Bh + \nabla \chi + \Bepsilon^{-1} {\rm curl}~\Bw,
\end{equation}
where $\Bh$, $\chi$, and $\Bw$ are all unique. Moreover, we have the estimate
\begin{equation}\label{5-21}
\|\Bh\|_{L^2(\Omega)} +  \|\nabla \chi\|_{L^2(\Omega)} + \|\Bw\|_{W^{1, 2}(\Omega)} 
\leq C \| \Bu \|_{L^2(\Omega)}. 
\end{equation}
\end{theorem}

\begin{proof}
The scalar potential $\chi \in W_0^{1, 2}(\Omega)$ is taken as a variational solution of the following problem 
\begin{equation} \nonumber
\forall \ \varphi \in W_0^{1, 2}(\Omega), \ \ \ \int_{\Omega} \Bepsilon \nabla \chi \cdot \nabla \varphi \, dx 
= \int_{\Omega} \Bepsilon \Bu \cdot \nabla \varphi \, dx. 
\end{equation}
Such a scalar function $\chi$ is unique, due to Lax-Milgram theorem. And it holds that 
\begin{equation}
\|\nabla \chi \|_{L^2(\Omega)} \leq C \| \Bu \|_{L^2(\Omega)}. 
\end{equation}
Moreover, the variational equality implies that 
\begin{equation}\nonumber
{\rm div}~(\Bepsilon u - \Bepsilon \nabla \chi ) = 0,\ \ \ \mbox{in}\ \Omega.
\end{equation}

Next, the $\Bepsilon$-harmonic field $\Bh$ is chosen in the following way, 
\begin{equation}\nonumber
\left\{ \begin{array}{l}
{\rm curl }~\Bh = 0, \ \ \ \mbox{in}\ \Omega,\\[2mm]
{\rm div}~(\Bepsilon \Bh ) = 0,\ \ \ \mbox{in}\ \Omega, \\[2mm]
\Bh \times \Bn = 0,\ \ \ \mbox{on}\ \partial \Omega,\\[2mm]
< \Bepsilon \Bh \cdot \Bn, 1>_{\Gamma_i} = <(\Bepsilon \Bu - \Bepsilon \nabla \chi ) \cdot \Bn, 1>_{\Gamma_i} ,\ \ 1\leq i \leq N_1. 
\end{array}  \right.
\end{equation}
In fact, according to the proof of Theorem \ref{theorem4.3}, 
\begin{equation}
\Bh = \sum_{i =1}^{N_1} <(\Bepsilon \Bu - \Bepsilon \nabla \chi) \cdot \Bn, 1>_{\Gamma_i} \nabla q_i^N,
\end{equation}
where $\{ \nabla q_i^N ; \ 1 \leq i \leq N_1  \}$ is a basis for $K_{N, \Bepsilon}^2(\Omega)$, derived in Subsection 4.3. Hence, 
\begin{equation}
\|\Bh \|_{L^2(\Omega)} \leq  C \|\Bepsilon \Bu  - \Bepsilon \nabla \chi\|_{L^2(\Omega)} 
+ C \| {\rm div}~(\Bepsilon \Bu - \Bepsilon \nabla \chi)        \|_{L^2(\Omega)} 
\leq C \|\Bu \|_{L^2(\Omega)}, 
\end{equation}
due to the fact that ${\rm div}~(\Bepsilon \Bu  - \Bepsilon \nabla \chi ) =  0$ in $\Omega$. 

Let $\Bz = \Bepsilon ( \Bu - \Bh  -\nabla \chi)$, $\Bz \in L^2(\Omega)$. It easy to easy to check that 
\begin{equation}\nonumber 
\left\{
\begin{array}{l}
{\rm div}~\Bz = 0,\ \ \mbox{in}\ \Omega,\\[2mm]
< \Bz \cdot \Bn, 1>_{\Gamma_i} = 0,\ \ \ 0\leq i \leq N_1. 
\end{array}
\right.
\end{equation}
It follows from Lemma \ref{lemma3-1} that there exists a vector potential $\Bw\in W_{{\rm div}^0}^{1, 2}(\Omega) \cap W^{1, 2}_T(\Omega)$, such that 
\begin{equation}
\Bz = {\rm curl}~\Bw,\ \ \ \mbox{in}\ \Omega .
\end{equation}
Moreover, $\Bw$ is unique and we have the estimate
\begin{equation}\nonumber
\|\Bw\|_{W^{1, 2}(\Omega)} \leq C \|\Bz\|_{L^2(\Omega)} \leq C \|\Bu \|_{L^2(\Omega)}. 
\end{equation}

At last, let us talk about the uniqueness of the decomposition. Suppose that $\Bu$ has another decomposition, 
\begin{equation}\nonumber
\Bu =  \tilde{\Bh} + \nabla \tilde{\chi} + \Bepsilon^{-1} {\rm curl}~\tilde{\Bw} , 
\end{equation}
where $\tilde{\Bh}\in K_{N, \Bepsilon}^2(\Omega)$, $\tilde{\chi } \in W^{1,2}_0 (\Omega)$, and $\tilde{\Bw} \in W_{{\rm div}^0 }^{1,2 }(\Omega) \cap W_T^{1,2}(\Omega). $ Then
\begin{equation}\label{5-35}
 0 = ( \Bh - \tilde{\Bh}) + \nabla (\chi  - \tilde{\chi}) + \Bepsilon^{-1} {\rm curl}~( \Bw - \tilde{\Bw} ). 
\end{equation}
Taking the $L^2$-inner product of \eqref{5-35} and $\Bepsilon (\Bh - \tilde{\Bh})$, 
we have
\begin{equation}\nonumber
0 = \left< (\Bh - \tilde{\Bh} , \ \Bepsilon (\Bh - \tilde{\Bh} ) \right> .
\end{equation}
It implies that 
\begin{equation}\nonumber
\Bh = \tilde{\Bh},\ \ \mbox{in}\ \Omega. 
\end{equation}
Similarly, taking the $L^2$-inner product of \eqref{5-35} and $\Bepsilon \nabla(\chi - \tilde{\chi})$, ${\rm curl}~(\Bw -\tilde{\Bw})$
resp., we have 
\begin{equation}\nonumber
\nabla (\chi - \tilde{\chi}) = 0,\ \ \ \ \mbox{and}\ \ \ {\rm curl}~(\Bw - \tilde{\Bw}) = 0,\ \ \mbox{in}\ \Omega. 
\end{equation}
That completes the proof of  Theorem \ref{decomposition3}
\end{proof}

\begin{remark}
Theorem \ref{decomposition3} can be conculded in a simple formula, 
\begin{equation}\nonumber
L^2(\Omega) = K_{N, \Bepsilon}^2(\Omega) \oplus  \nabla W^{1, 2}_0(\Omega) \oplus \Bepsilon^{-1} {\rm curl}~\left(W^{1, 2}_{{\rm div}^0 }(\Omega) \cap W_{T}^{1, 2}(\Omega) \right). 
\end{equation}
\end{remark}

The above decomposition can be generalized to $W^{m, p}$-spaces. 

\begin{theorem}\label{decomposition4}
Let $\Bu \in W^{m, p}(\Omega)$, $m \in \mathbb{N}^*$, $1 < p < +\infty$. The decomposition in Theorem \ref{decomposition3} satisfies the following estimate,
\begin{equation}\label{5-41}
\|\Bh \|_{W^{m, p}(\Omega)} + \|\nabla \chi \|_{W^{m, p}(\Omega)} + \|\Bw\|_{W^{m+1, p}(\Omega)} 
\leq C \| \Bu \|_{W^{m, p}(\Omega)}. 
\end{equation}
\end{theorem}

\begin{proof}
When $\Bu \in W^{m, p}(\Omega)$, $\chi$ is in fact a solution of the system 
\begin{equation}\nonumber
\left\{ \begin{array}{l}
{\rm div}~(\Bepsilon \nabla \chi ) = {\rm div}~(\Bepsilon \Bu),\ \ \ \mbox{in}\ \Omega, \\[2mm]
\chi = 0, \ \ \ \mbox{on}\ \partial \Omega. 
\end{array}
\right.
\end{equation}
It follows from the classical regularity theory for elliptic equation with Dirichlet boundary condition\cite{GT}, that
\begin{equation}\nonumber
\|\nabla \chi \|_{W^{m, p}(\Omega)} \leq C \| \Bu \|_{W^{m, p}(\Omega)}. 
\end{equation}

On the other hand, 
\begin{equation}\nonumber
\| \Bh \|_{W^{m, p}(\Omega)} \leq C \sum_{i =1}^{N_1} 
\left|<(\Bepsilon \Bu - \Bepsilon \nabla \chi ) \cdot \Bn ,  \, 1>_{\Gamma_i}   \right|
\leq C \|\Bepsilon \Bu -\Bepsilon \nabla \chi \|_{W^{m, p}(\Omega)} 
\leq C \| \Bu \|_{W^{m, p}(\Omega)}.
\end{equation}
And the estimate for $\Bw$ follows from Lemma \ref{lemma3-1}. 
\end{proof}

\begin{remark}
It is still open to us whether the estimates \eqref{5-11} and \eqref{5-41} hold when $m =0$, $p\neq 2$. 
\end{remark}

{\bf Acknowledgment}\ \ The work of the author was partially supported by NSFC grant No. 11671289.

\end{document}